\documentclass[a4paper,11pt]{article}

\usepackage{amsmath, amscd}
\usepackage{fullpage}

\usepackage[all]{xy}
\usepackage{amsthm}

\usepackage{tikz}
\usetikzlibrary{3d}
\usepackage{amssymb}
\usepackage{latexsym}
\usepackage{enumerate}
\usepackage{graphicx}

\usepackage[utf8]{inputenc}

\newtheorem{Theorem}{Theorem}[section]

\newtheorem{Definition}[Theorem]{Definition}
\newtheorem{Proposition}[Theorem]{Proposition}
\newtheorem{Lemma}[Theorem]{Lemma}
\newtheorem{Corollary}[Theorem]{Corollary}
\newtheorem{Remark}[Theorem]{Remark}

\newtheorem{theorem}[Theorem]{Theorem}
\newtheorem{proposition}[Theorem]{Proposition}
\newtheorem{lemma}[Theorem]{Lemma}
\newtheorem{corollary}[Theorem]{Corollary}

\theoremstyle{definition}
\newtheorem{remark}[Theorem]{Remark}
\theoremstyle{definition}

\usepackage{comment}
\usepackage{transparent}
\usepackage{calc} 

\title{The deformation space of non-orientable hyperbolic 3-manifolds}
\author{Juan Luis Durán Batalla and Joan  Porti
}
\date{\today}

\begin{document}

\maketitle

\begin{abstract}
We consider non-orientable hyperbolic 3-manifolds of finite volume $M^3$. 
When $M^3$ has an ideal triangulation $\Delta$, 
we compute the deformation space of the pair $(M^3, \Delta)$ (its Neumann Zagier parameter space).
We also determine the variety of representations of $\pi_1(M^3)$ in 
$\mathrm{Isom}(\mathbb{H}^3)$ in a neighborhood of the holonomy. As a consequence, when some ends are non-orientable, there are
deformations from the 
variety of representations that cannot be realized as deformations of the  pair $(M^3, \Delta)$.
We also discuss the metric completion of these structures and we 
illustrate the results on the Gieseking manifold.
\end{abstract}

\section{Introduction}
Let $M^3$ be a complete,  non-compact, hyperbolic three-manifold of finite volume. Assume  first 
that $M^3$ is orientable. 
Assume also that $M^3$ has a geometric ideal triangulation $\Delta$ (see \cite{NeumannZagier} for the definition). 
Following Thurston's construction for the figure eight knot
exterior
in \cite{ThurstonNotes}, Neumann and Zagier defined in \cite{NeumannZagier} a deformation space
of the pair $(M^3, \Delta)$, by considering the set 
of parameters of the ideal simplices of $\Delta$
subject to compatibility equations. We denote the Neumann-Zagier parameter space by $\mathrm{Def}(M^3, \Delta)$. 
It is proved in  
 \cite{NeumannZagier} that it is homeomorphic to an open subset of $\mathbb{C}^l$, where $l$ is the
 number of ends of $M^3$.
 
Another approach to deformations is based on $\mathcal{R}(\pi_1(M^3)
,
\mathrm{Isom}(\mathbb{H}^3))$, the variety of conjugacy classes of representations
of $\pi_1(M^3)$ in $\mathrm{Isom}(\mathbb{H}^3)$. It is proved for instance by Kapovich in \cite{KapovichBook}  that 
a neighborhood of the holonomy of $M^3$ is bi-analytic to an open subset of $\mathbb{C}^l$. 

Both approaches to deformations can be used to prove the hyperbolic Dehn filling theorem 
(even if it is still an open question whether 
$M^3$ orientable admits a geometric ideal triangulation or not). 
Among other things, one has to take into account that $\mathrm{Def}(M^3, \Delta)$
is a $2^l$ to $1$ branched covering of the neighborhood in 
$\mathcal{R}(\pi_1(M^3)
,
\mathrm{Isom}(\mathbb{H}^3))$.
When $M^3$ is orientable, 
both approaches yield the same deformation space.

In this paper we investigate the non-orientable setting, that is, $M^3$  is a connected, 
non-orientable, hyperbolic $3$-manifold of finite volume. When it has an ideal triangulation $\Delta$,
we define a deformation space of the pair $\mathrm{Def}(M^3, \Delta)$ \`a la Neumann-Zagier.
Here is the main result of the paper (for simplicity, we assume that $M^3$ has a single end, that is non-orientable):

\begin{theorem}
\label{Thm:main_thm_one_cusp}
        Let $M^3$ be a complete non-orientable hyperbolic $3$-manifold of finite
volume with a single end, that is non-orientable. 
\begin{enumerate}[(a)]
\item If $M^3$ admits a geometric ideal triangulation
$\Delta$, then, $\mathrm{Def}(M^3, \Delta)\cong (-1,1)$, where  the  parameters $\pm t\in (-1,1)$ 
correspond to the same structure.

\item A neighborhood of the holonomy in $\mathcal{R}(\pi_1(M^3),\mathrm{Isom}(\mathbb{H}^3))$
is homeomorphic to an interval $(-1,1)$.
\end{enumerate}
Furthermore,   the holonomy map
$\mathrm{Def}(M^3, \Delta)\to \mathcal{R}(\pi_1(M^3),\mathrm{Isom}(\mathbb{H}^3))$  
folds the interval $(-1,1)$ at 0 and its image is
the  half-open interval $[0,1)$, where $0$ corresponds to the complete structure.
\end{theorem}

The version with \emph{several cusps} of this theorem is stated in
Theorem~\ref{Thm:main_thm_several_cusps}.

For $M^3$ as in the theorem, structures in the subinterval $[0,1)\subset (-1,1)$ in the variety of representations
 are realized by $\mathrm{Def}(M^3, \Delta)$, but structures in $(-1,0)$ are not.
This corresponds to two different kinds of representations of the Klein bottle, that we denote of type I
when realized, and II when not. Those are described in Section~\ref{S:RepKl}. 

Deformations of the complete structure are non-complete, and therefore for a deformation of the holonomy 
 the hyperbolic structure is not unique.
Deformations of type I can be realized by  ideal triangulations, hence there is a natural choice of structure
and we prove
in Theorem~\ref{Theorem:completion} that its metric completion consist in adding a singular geodesic, so that
it is the core of a solid Klein bottle and it has a singularity of cone angle in a neighborhood of zero.
For deformations of type II, we prove  that there is a natural choice of structure (radial),
and  the metric completion consists in adding a singular interval, also in
Theorem~\ref{Theorem:completion}. 
This singular interval is the soul of a twisted disc orbi-bundle over an interval with mirror boundary.
Topologically, a neighborhood of this interval is the disc sum of two cones on a projective plane. Metrically, it is 
a conifold, with cone angle at the interior of the interval in a neighborhood of zero.
 
The paper is organized as follows. In Section~\ref{S:CombinatorialDef} we describe $\mathrm{Def}(M^3, \Delta)$ and in 
Section~\ref{S:VarReps},   $\mathcal{R}(\pi_1(M^3),\mathrm{Isom}(\mathbb{H}^3))$.  Section~\ref{S:RepKl} is devoted 
to  representations of the Klein bottle. Metric completions are described in Section~\ref{S:Completion}, and finally 
in Section~\ref{S:Gieseking} we describe in detail the deformation space(s) of the Gieseking manifold.

\paragraph{Acknowledgement} We thank the anonymous referee for useful suggestions.
The first author is supported by doctoral grant BES-2016-079278. Both authors are partially supported by the 
Spanish State Research Agency through   grant
PGC2018-095998-B-I00, as well as 
the Severo Ochoa and María de Maeztu Program for Centers and Units of Excellence in R\&D (CEX2020-001084-M).

\section{Deformation space from ideal triangulations}
\label{S:CombinatorialDef}

Before discussing non-orientable manifolds, we recall first the orientable case. 
The first example was constructed by Thurston in his notes \cite[Chapter 4]{ThurstonNotes}
for the figure eight knot exterior, and the general case was constructed by 
Neumann and Zagier in 
\cite{NeumannZagier}.  We refer the reader to these references for the upcoming exposition.

From the point of view of a triangulation, the deformation of the hyperbolic 
structure on a manifold with a given \emph{geometric ideal} triangulation is the 
space of parameters of ideal tetrahedra, subject to compatibility equations.

 A  \emph{geometric ideal} 
tetrahedron is a geodesic tetrahedron of $\mathbb{H}^3$ with   all of its 
vertices in the ideal sphere $\partial_\infty\mathbb{H}^3$. We say 
that a hyperbolic $3$-manifold \emph{admits a geometric ideal triangulation} if it
is the union of such tetrahedra, along the geodesic faces. 
Though it has been established in many cases, it is still an
open problem to decide whether every orientable hyperbolic three-manifold of finite volume admits 
a geometric ideal triangulation.

Given an ideal tetrahedron in $\mathbb{H}^3$, up to isometry
we may assume that its ideal vertices in $\partial_\infty\mathbb H^3\cong\mathbb C\cup\{\infty \}$
are $0$, $1$, $\infty$ and $z\in \mathbb{C}$. The idea of Thurston is to equip the 
(unoriented) edge between $0$ and $\infty$ with the complex number $z$,
called 
\textit{edge invariant}. The edge invariant determines the isometry class of the tetrahedron, and
for different edges the corresponding invariants satisfy some relations, called 
\textit{tetrahedron relations}:
\begin{itemize}
\item Opposite edges have the same invariant.
\item Given $3$ edges with a common end-point and invariants $z_1, z_2, z_3$, 
indexed following the right hand rule towards the common ideal vertex, they are 
related to $z_1$ by $z_2=\frac{1}{1-z_1}$ and $z_3=\frac{z_1-1}{z_1}$. 
\end{itemize}

Let $M^3$ be a possibly non-orientable complete hyperbolic $3$-manifold of finite 
volume, which admits a geometric ideal triangulation, $\Delta= \{A_1, \cdots, 
A_n\}$ . As we have stated before, up to (oriented) isometry, the hyperbolic structure of each tetrahedron can be determined by a single edge invariant, thus the usual parameterization of the triangulation goes as follows: we fix 
an edge $e_{i}$ in each tetrahedron $A_i$, and consider its edge invariant, 
$z_i$.  Hence, the hyperbolic structure of $M^3$ can be parametrized by $n$ parameters (one for each tetrahedron) and we will denote the parameters of the complete triangulation 
$\{z_1^0,\cdots, z_n^0\}$. The \textit{deformation space of $M^3$ with respect to 
$\Delta$}, $\mathrm{Def}(M^3, \Delta)$, is defined as the set of parameters 
$\{z_1,\cdots, z_n\}$ in a small enough neighborhood of the complete structure 
for which the gluing bestows a hyperbolic structure on $M^3$. However, we find 
that the equations defining the deformation space are easier to work with if we 
use $3n$ parameters (one for each edge after taking into account the duplicity in 
opposite edges) and ask them to satisfy the second tetrahedron relation too. 

When $M^3$ is orientable, in order for the gluing to be geometric it 
is necessary and sufficient that around each edge cycle $[e]=\{e_{i_1,j_1}, 
\cdots, e_{i_n, j_n}\}$ the following two compatibility conditions are 
satisfied:
\begin{gather}
\label{equation_edge_orientable}
\prod_{l=1}^n z(e_{i_l,j_l})=1, \\
\label{equation_edge_orientable2}
\sum_{l=1}^n \arg(z(e_{i_l,j_l}))=2 \pi.
\end{gather}
The geometric meaning of these equations can be seen as follows: if we try to realize in $\mathbb{H}^3$ the tetrahedra around the edge cycle $[e]$, Equation~\eqref{equation_edge_orientable} means that the triangulation must `close up' and Equation~\eqref{equation_edge_orientable2}, that the angle around $[e]$ must be precisely $2\pi$ (instead of a multiple). The parameters of the complete hyperbolic structure are denoted by 
$\{z^0(e_{1,1}), 
\cdots, z^0(e_{n,3})\}$. In a  small enough neighborhood of
$\{z^0(e_{1,1}), 
\cdots, z^0(e_{n,3})\}$,  fulfillment of
\eqref{equation_edge_orientable} implies \eqref{equation_edge_orientable2}.
We end the overview of the orientable case with the theorem 
we want to extend to the non-orientable case:

\begin{theorem}[Neumann--Zagier \cite{NeumannZagier}]
\label{thm:neumannzagier-thurston}
Let $M^3$ be connected, oriented, hyperbolic, of finite volume with $l$ cusps. 
Then $\mathrm{Def}(M^3, \Delta)$ is bi-holomorphic to an open set of $\mathbb{C}^l$.
\end{theorem}

When we  deal  with non-orientable manifolds, again the problem of the 
gluing being geometric lives within a neighborhood of the edges. The compatibility equations in this case carry the same geometric meaning as in Equations~\eqref{equation_edge_orientable} and \eqref{equation_edge_orientable2}, but accounting for the possible change of orientation of the tetrahedra.

\begin{proposition}
Let $M^3$ be a non-orientable manifold triangulated by a finite number of ideal 
tetrahedra $A_i$. The triangulation bestows a hyperbolic structure around the 
edge cycle $[e]=\{e_{i_1,j_1}, \cdots, e_{i_n,j_n}\}$ if and only if the 
following compatibility equations are satisfied:
\begin{gather}
\label{compeq1}
\prod_{l=1}^n 
\frac{z(e_{i_l,j_l})^{\epsilon_l}}{\overline{z(e_{i_l,j_l})}^{1-\epsilon_l}}=1, 
\\
\label{compeq2}
\sum_{l=1}^n \arg(z(e_{i_l,j_l}))=2 \pi,
\end{gather}
where $z(e_{i_l,j_l})$ is the edge invariant of $e_{i_l,j_l}$, and 
$\epsilon_l=0,1$ in such a way that, in the gluing around the edge cycle $[e]$, 
a coherent orientation of the tetrahedra is obtained by gluing a copy of 
$A_{i_l}$ with its orientation reversed if $\epsilon_l=0$, (or kept the original 
one if $\epsilon_l=1$), and with the initial condition that the orientation of 
the tetrahedron $A_{i_1}$ is kept as given.
\end{proposition}

\begin{proof} When we follow a cycle of side identifications around an edge,
can always reorient the tetrahedra (maybe more than once) so 
that the gluing is done by orientable isometries. The compatibility equations 
for the orientable case can be then applied and, hence, for the neighborhood of 
the edge cycle to inherit a hyperbolic structure, 
\eqref{equation_edge_orientable} must be satisfied, with the corresponding edge 
invariants.

Now, let us consider an edge $e_{i,j}\in A_i$ with parameter $z(e_{i_j})$. To 
see how the edge invariant changes under a non-orientable isometry, we can 
assume  that $A_{i}$ has vertices $0, 1, z(e_{i,j})$ 
and $\infty$ in the upper-space model, and consider the isometry is $c$,
the Poincar\'e extension of the complex 
conjugation in 
$\partial_{\infty}\mathbb H^3\cong\mathbb{C}\cup\{\infty\}$. Then, the edge invariant of $c(e_{i,j})\in c(A_i)$ is 
${1}/{\overline{z(e_{i,j})}}$.

Thus, the proposition follows with ease after changing the orientation of some 
tetrahedra.
\end{proof}

\begin{Definition}
\label{dfn:DefSpace}
Let $M^3$ be a connected, complete, non-orientable, hyperbolic 3-manifold of 
finite volume. Let $\Delta$ be an ideal triangulation of $M^3$. The \emph{deformation 
space of $M^3$ related to the triangulation $\Delta$} is
\[
\begin{array}{rl}
\mathrm{Def}(M^3, \Delta)= & \{(z_{1,1}, \cdots, z_{n,3})\in U\cap\mathbb{C}^{3n} 
\textrm{ satisfying the compatibility} \\
& \textrm{ equations \eqref{compeq1} and \eqref{compeq2} and the tetrahedron 
relations} \},
\end{array}
\]
where $U$ is a small enough neighborhood of the parameters $(z^0_{i,j})$ of the 
complete structure.
\end{Definition}

Let $M^3_+$ be the orientation covering of $M^3$. The ideal triangulation on 
$M^3$, $\Delta$, can be lifted to an ideal triangulation $\Delta_+$ on 
$M^3_+$. There is an orientation reversing homeomorphism, $\iota$, acting on 
$M^3_+$ such that $M^3=M^3_+/\iota$ and $\iota^2=\mathrm{Id}$. The triangulation on 
$M^3_+$ is constructed in the usual way: for every tetrahedron $A_i$  we take 
another tetrahedron with the opposite orientation, $\iota(A_i)$, and glue them 
so that the orientation is coherent.  For every edge, $e_{i,j} \in A_i$, let 
$z(e_{i,j})$ or $z_{i,j}$ denote its edge invariant. Analogously, 
$w(\iota(e_{i,j}))$ or $w_{i,j}$ will denote the edge invariant of 
$\iota(e_{i,j}) \in \iota(A_i)$.

\begin{remark}
\label{rmk_non-orientable_equation}The compatibility equations \eqref{compeq1} 
and \eqref{compeq2} around $[e]\in M^3$ are precisely the (orientable) 
compatibility equations in any lift of $[e]$ to the orientation covering.
\end{remark}

The orientation reversing homeomorphism acts on $\mathrm{Def}(M^3_+, 
\Delta_+)$ by pulling-back (equivalently, pushing-forward) the associated 
hyperbolic metric on each tetrahedron. Combinatorially, the action is described 
in the following lemma:

\begin{lemma}
\label{lemma_iotaacts}
Let $M^3=M^3_+/\iota$, where $\iota$ is an orientation reversing homeomorphism. 
Let $M^3$ admit an ideal triangulation $\Delta$. Then, $\iota$ acts on 
$\mathrm{Def}(M^3_+, \Delta_+)$ as
\begin{equation}
	\label{iotaacts}
	\iota_*((z_{i,j},w_{i,j}))=(\frac{1}{\overline{w_{i,j}}}, 
\frac{1}{\overline{z_{i,j}}}).
	\end{equation} 
\end{lemma}

\begin{proof}
The proof follows easily from the fact that $\iota$ permutes the edges and, for 
$e_{i,j}\in A_i$ with invariant 
$z(e_{i,j})$, the edge invariant of $c(e_{i,j}) \in c(A_i)$ is 
$\frac{1}{\overline{z(e_{i,j})}}$, where $c$ is the Poincar\'e extension of the complex 
conjugation.
\end{proof}

\begin{remark} Metrics on tetrahedra are considered up to isotopy.
\end{remark}

\begin{corollary}
\label{coro:isomorphism_def_orientation_cover}

The map $(z_{i,j}) \in \mathrm{Def}(M^3,\Delta) \longmapsto (z_{i,j}, 
1/\overline{z_{i,j}})\in \mathrm{Def}(M^3_+, \Delta_+)^\iota$ is a real 
analytic isomorphism.
\end{corollary}

\begin{proof} It follows from Remark~\ref{rmk_non-orientable_equation} and Lemma~\ref{lemma_iotaacts}.
\end{proof}

Our goal is to use Corollary~\ref{coro:isomorphism_def_orientation_cover} and Theorem~\ref{thm:neumannzagier-thurston} 
in order to identify the deformation space of 
$M^3$ with the fixed points under an action on $\mathbb{C}^k$. Let us suppose for 
the time being that $M^3$ has only one cusp which is non-orientable. The section 
of this cusp must be a Klein bottle. In order to define the bi-holomorphism 
through generalized Dehn filling coefficients we must first fix a longitude-meridian pair in the 
peripheral torus in the orientation covering $M^3_+$. As we will see, there is 
a canonical choice. Afterwards, following Thurston, we will compute the derivative of the holonomy, 
$\mathrm{hol'}$, and translate the action of $\iota$ over there and finally, to 
the generalized Dehn filling coefficients.

\bigskip

\noindent \textbf{Fixing a longitude-meridian pair.} Let $K^2$ be Klein bottle, 
its fundamental group admits a presentation 
$$\pi_1(K^2)=\langle a, b | aba^{-1}=b^{-1}\rangle $$
The elements $a^2, b$ in the orientation covering $T^2$ are generators of 
$\pi_1(T^2)$ and are represented by the unique homotopy classes of loops in the 
orientation covering that are invariant by the deck transformation (as 
unoriented curves). From now on, we will choose as longitude-meridian pair 
the elements:
\begin{align*}
 l&:=a^2, \\ m&:=b.
\end{align*}

\begin{Definition}
\label{def:distiguished}
The previous generators of $\pi_1(T^2)$ are called \emph{distinguished} 
elements.
\end{Definition}

\begin{lemma}
Let $[\alpha] \in \pi_1(T)$, let $\iota$ be the involution in the orientation 
covering $M^3_+$, that is, $M^3\cong M^3_+/\iota$. We also denote by $\iota$ 
the restriction of $\iota$ to the peripheral torus $T$. If 
$$\mathrm{hol}'(\alpha)=\prod_{r\in I}z(e_{i_r,j_r})^{\epsilon_r} \prod_{s\in J} 
w(\iota(e_{i_s,j_s}))^{\epsilon_s},$$
where $\epsilon_r, \epsilon_s \in \{\pm 1\} $, then 
$$\mathrm{hol'}(\iota(\alpha))=\prod_{r \in 
I}w(\iota(e_{i_r,j_r}))^{-\epsilon_r} \prod_{s\in J} 
z(e_{i_s,j_s})^{-\epsilon_s}.$$
\end{lemma}

When we compute the derivative of the holonomy of an element, $\mathrm{hol}'(\gamma)$, we assume that $\mathrm{hol}(\gamma)$ fixes 
$\infty$.

\begin{proof} We use Thurston's method for computing the holonomy 
through the developing of triangles in $\mathbb{C}$ (see \cite{ThurstonNotes}). 
Thus, the factor that each piece of path adds to the derivative of the holonomy 
changes as in Figure~\ref{fig:action} under the action of $\iota$. 
\end{proof}

\begin{figure}[h!]
	\centering
	\def\svgwidth{0.5\textwidth}
	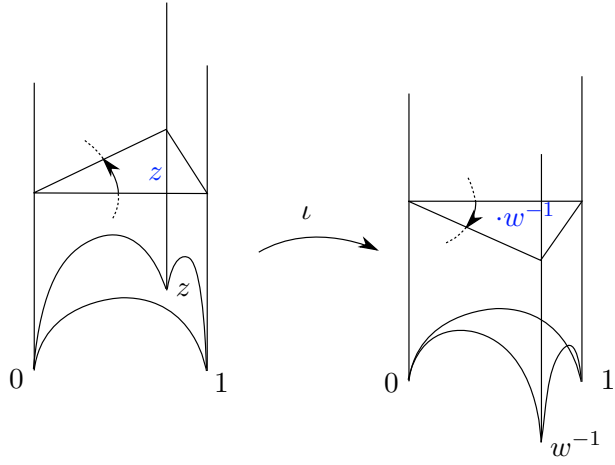
	\caption{Change under the action of $\iota$.}
	\label{fig:action}
\end{figure}

\begin{proposition}
\label{propiotaaccts2}
For the chosen longitude-meridian pair, the action of $\iota$ on 
$\mathrm{Im}(\mathrm{hol'})\subset \mathbb{C}^2$  is
\begin{equation}
\label{iotaacts2}
\iota_*(L,M)=(\overline{L}, \overline{M}^{-1}),
\end{equation}
where $L=\mathrm{hol'}(l)$, $M=\mathrm{hol'}(m)$.
\end{proposition}

\begin{proof}
 The action of $\iota$ on the longitude-meridian pair is $\iota_*(l)=l$, 
$\iota_*(m)=m^{-1}$. Hence, the previous lemma implies that the derivative holonomy of 
the longitude and the meridian has the following features:
\begin{gather}
\mathrm{hol'}(m)=\prod_{r\in I}z(e_{i_r,j_r})^{\epsilon_r} \prod_{s\in J} 
w(\iota(e_{i_s,j_s}))^{\epsilon_s}=\prod_{r \in 
I}w(\iota(e_{i_r,j_r}))^{\epsilon_r} \prod_{s\in J} z(e_{i_s,j_s})^{\epsilon_s}, 
\\
\mathrm{hol'}(l)=\prod_{r\in I } (z(e_{i_r,j_r}) 
w(\iota(e_{i_r,j_r}))^{-1})^{\epsilon_r}.
\end{gather}

\end{proof}

\begin{remark}
\label{rmk:hol_lm} Following the notation of Proposition~\ref{propiotaaccts2}, 
$(L,M) \in (\mathbb{C}^2)^\iota$ if and only if $L\in \mathbb{R}$, $|M|=1$.
\end{remark}

Let us denote by $u:=\log \mathrm{hol'}(l)$, $v:=\log \mathrm{hol'}(m)$, the \emph{generalized
Dehn coefficients} are the solutions in $\mathbb{R}^2 \cup \{\infty \}$ to 
 \emph{Thurston's equation}
\begin{equation}
\label{dehn_equation}
pu+qv=2\pi i.
\end{equation}
Indeed, Neumann-Zagier Theorem~\ref{thm:neumannzagier-thurston} (see also~\cite{ThurstonNotes})
states that, for $M^3$ orientable, the map $(z_{i,j}) \in \mathrm{Def}(M^3, \Delta) 
\mapsto (p_k,q_k)$ is a bi-holomorphism, and the image is a neighborhood of 
$(\infty, \cdots, \infty)\in \overline{\mathbb{C}}^l$, where $l$ is the 
number of cusps of $M^3$.
 
\begin{proposition}
The action of $\iota$ on $(p,q) \in U\cap \mathbb{R}^2\cup \{\infty\}$, where 
$(p,q)$ are the generalized Dehn coefficients, is
\begin{equation}
\label{iotaacts4}
\iota_*(p,q)=(-p,q).
\end{equation}
\end{proposition}

\begin{proof}

The action of $\iota$ can be translated through the logarithm to $(u,v)$ from 
the action on the holonomy \eqref{iotaacts2} as 
$\iota_*(u,v)=(\overline{u},-\overline{v})$. Then, to find the action on generalized
Dehn coefficients, we have to solve Thurston's equation \eqref{dehn_equation} with 
$\overline{u}$ and $-\overline{v}$, that is,
\begin{equation}
p'\overline{u}-q'\overline{v}=2\pi i,
\end{equation}
where $\iota_*(p,q)=(p',q')\in \mathbb{R}^2\cup \{\infty \}$. It is 
straightforward to check that $(p',q')=(-p,q)$ is the solution.
\end{proof}

\begin{corollary}
The fixed points under $\iota$, which are in correspondence with 
$\mathrm{Def}(M^3, \Delta)$, are those whose generalized Dehn filling coefficients are of type 
$(0,q)$.
\end{corollary}

\begin{theorem}
	\label{deformation}
	Let $M^3$ be a connected, complete, non-orientable, hyperbolic 3-manifold 
of finite volume. Let $M^3$ have $k$ non-orientable cusps and $l$ orientable ones 
and let it admit an ideal triangulation $\Delta$. Then $\mathrm{Def}(M^3,\Delta)$ 
is real bi-analytic  to an open set of $\mathbb{R}^{k+2l}$.
\end{theorem}

\begin{proof}
 We have already proved the theorem for $k=1, l=0$. 

Let $k=0, l=1$. Any peripheral torus on $M^3$ is lifted to two peripheral tori, 
$T_1, T_2$, on $M^3_+$. The action of $\iota$ here is by permutation. More 
precisely, when can fix any longitude-meridian pair in one, $l_1, m_1 \in 
\pi_1(T_1)$, and choose the longitude-meridian pair in the second torus as 
$l_2:=\iota_*(l_1), m_2:=\iota_*(m_1) \in \pi_1(T_2)$. The same arguments as in 
Proposition~\ref{propiotaaccts2} show that $\iota_*(p_1,q_1, p_2, q_2) 
=-(p_2,q_2,p_1,q_1)$ and, hence, the fixed points have generalized coefficients 
$(p,q,-p,-q)$, $p,q \in \mathbb{R}$.

Finally, in general the action of $\iota$ on $\mathrm{Im}(\mathrm{hol'})\subset 
\mathbb{C}^{k+2l}$ can be understood as a product of $k+l$ actions, 
$\iota_1\times \cdots \times \iota_l$, the first $k$, $\iota_i, i=1, \cdots , k$ 
acting on $\mathbb{C}$ as in the case for a Klein bottle cusp, and the 
subsequent $l$, $\iota_j, j=k+1, \cdots , k+l$, acting on $\mathbb{C}^2$ as in 
the case for a peripheral torus. 
\end{proof}

\section{Varieties of representations}
\label{S:VarReps}
The group of isometries of hyperbolic space is denoted by $G$, and we have the following well known isomorphisms:
$$
 G=\mathrm{Isom}(\mathbb{H}^3)\cong\mathrm{PO}(3,1)\cong\mathrm{PSL}(2,\mathbb{C})\rtimes \mathbb{Z}_2,
$$
which we will often consider in order to identify elements of $G$ with elements of $\mathrm{PSL}(2,\mathbb{C})\rtimes \mathbb{Z}_2$. The group $G$ has two connected components, according to whether the isometries preserve or reverse the orientation.

For a finitely generated group $\Gamma$, the variety of representations of $\Gamma$ in $G$ is denoted by
$$
\hom(\Gamma, G).
$$
As $G$ is algebraic, it has a natural structure of algebraic set, cf Johnson--Millson \cite{JohnsonMillson}, 
but we  consider only its topological structure.
We are interested in the set of conjugacy classes of representations: 
$$
\mathcal{R}(\Gamma,G)=\hom(\Gamma, G)/G.
$$
When $M^3$ is hyperbolic, we write $\Gamma=\pi_1(M^3)$. 
The holonomy of $M^3$ $$\mathrm{hol}\colon\Gamma\to G$$
is well defined up to conjugacy, hence $[\mathrm{hol}]\in \mathcal{R}(\Gamma,G)$.
To understand deformations, we analyze a neighborhood of the holonomy in 
$\mathcal{R}(\Gamma,G)$. 
The main result of this section is:

\begin{Theorem}
\label{Thm:dimdef}
Let $M^3$   be a hyperbolic manifold of finite volume. Assume that it has $k$ non-orientable cusps and $l$
orientable cusps. Then there exists a neighborhood of $[\mathrm{hol}]$ in  $\mathcal{R}(\Gamma,G)$ homeomorphic
to $\mathbb{R}^{k+2l}$. 
 \end{Theorem}

When $M^3$ is orientable this result is well known, see for instance Boileau--Porti or Kapovich
\cite{BoileauPorti,KapovichBook}, 
hence we assume that $M^3$ is non-orientable.
We will prove a more precise result in Theorem~\ref{Thm:coordinatesNO}, as 
for our purposes it is relevant to describe local coordinates in terms of the geometry of
holonomy structures at the ends.

Before starting the proof, we need a  lemma on varieties of representations. 
The projection to the quotient  $\pi\colon\hom(\Gamma,G)\to \mathcal{R}(\Gamma,G)$ can have quite bad properties,
for instance even if  $\hom(\Gamma,G)$ is Hausdorff, in general $\mathcal{R}(\Gamma,G)$ is not. But
in  a neighborhood of the holonomy we have:

\begin{Lemma}
\label{Lemma:GoodNbhd}
 There exists a neighborhood $V\subset \mathcal{R}(\Gamma,G)$ of $[\mathrm{hol}]$ such that:
 \begin{enumerate}[(a)]
  \item If $[\rho]=[\rho']\in V$, then the matrix $A\in G$ satisfying 
  $A\rho(\gamma) A^{-1}=\rho'(\gamma)$, $\forall\gamma\in\Gamma$, is \emph{unique}.
  \item $V$ is Hausdorff and the projection $\pi\colon \pi^{-1}(V)\to V$ is open.
  \item If $[\rho]\in V$, then $\forall\gamma\in \Gamma$ $\rho(\gamma)$ preserves the orientation
  of $\mathbb{H}^3$ if and only if $\gamma$ is represented by a loop that preserves the orientation of $M^3$
 \end{enumerate}
\end{Lemma}

Assertions (a) and (b) are proved for instance by Johnson and Millson in \cite{JohnsonMillson}:
they  define the property of \emph{good}
representation, that is open in $\mathcal{R}(\Gamma,G)$, it implies Assertions (a) and (b), 
and it is satisfied by the conjugacy class of the
holonomy. Assertion (c) is clear by continuity and the decomposition of $G$ in two components, according to the 
orientation.

To describe the neighborhood of the holonomy in $\mathcal{R}(\Gamma, G)$ we use the orientation covering.

\subsection{Orientation covering and the involution on representations}

As mentioned, we assume $M^3$ non-orientable.
Let 
$$M_+^3\to M^3$$ denote the orientation covering, with fundamental group $\Gamma_+=\pi_1(M^3_+)$.
In particular we have a short exact sequence:
$$
1\to\Gamma_+\to\Gamma\to\mathbb Z_2\to 1.
$$

\begin{Definition}
\label{dfn:sigma}
 For $\zeta\in \Gamma\setminus \Gamma_+$, define the group automorphism
 $$
 \begin{array}{rcl}
  \sigma_*\colon \Gamma_+ & \to & \Gamma_+ \\
   \gamma & \mapsto & \zeta\gamma\zeta^{-1}
 \end{array}
 $$
\end{Definition}

The automorphism $\sigma_*$ depends on the choice of  $\zeta\in \Gamma\setminus \Gamma_+$:
 automorphisms corresponding to 
 different choices of $\zeta$  differ by composition (or pre-composition) with an inner automorphism
of $\Gamma_+$; furthermore $\sigma_*^2$ is an inner automorphism because $\zeta^2\in\Gamma_+$.
This automorphism $\sigma_*$ is the map induced by the deck transformation of the
orientation covering $M^3_+\to M^3$.

The map induced by  $\sigma_*$ in the variety of representations is denoted by
$$
\begin{array}{rcl}
\sigma^*\colon \mathcal{R}(\Gamma, G) & \to & \mathcal{R}(\Gamma, G)\\
{}[\rho] & \mapsto & [\rho\circ\sigma_*]
\end{array}
$$
and $\sigma^*$ does not depend on the choice of $\zeta$, because $\sigma_*$ is well defined up to inner automorphism.
Furthermore  $\sigma^*$ is an involution, $(\sigma^*)^2=\mathrm{Id}$.

Consider the restriction map:
$$
 \mathrm{res}\colon \mathcal{R}(\Gamma,G) \mapsto \mathcal{R}(\Gamma_+,G)
$$
that maps the conjugacy class of a representation of $\Gamma$ to the conjugacy class of its 
restriction to $\Gamma_+$.

\begin{Lemma}
\label{Lemma:UV}
 There exist $U\subset \mathcal{R}(\Gamma, G)$ neighborhood of $[\mathrm{hol}]$ and
 $V\subset \mathcal{R}(\Gamma_+, G)$ neighborhood of $\mathrm{res}([\mathrm{hol}])$ such that
 $$
 \mathrm{res}\colon U\overset{\cong}\longrightarrow \{[\rho]\in V\mid \sigma^*([\rho])=[\rho]\}
 $$
 is a homeomorphism.
\end{Lemma}

\begin{proof}
 We show first that 
 $\mathrm{res}(\mathcal{R}(\Gamma, G))\subset   \{[\rho]\in  \mathcal{R}(\Gamma_+, G) \mid \sigma^*([\rho])=[\rho]\}  $:
 if $\rho_+=\mathrm{res}(\rho)$, then $\forall\gamma\in \Gamma_+$
 $$
 \sigma^*(\rho_+)(\gamma)=\rho_+(\sigma_*(\gamma))=\rho_+(\zeta\gamma\zeta^{-1})=
 \rho(\zeta)\rho_+(\gamma)\rho(\zeta)^{-1}.
 $$
Hence $\sigma^*([\mathrm{res}(\rho)])= [\mathrm{res}(\rho)]$.

Next, given $[\rho_+]\in   \mathcal{R}(\Gamma_+, G) $ satisfying  $\sigma^*([\rho_+])=[\rho_+]$, by construction 
there exists 
$A\in G$ that conjugates $\rho_+$  and $\rho_+\circ\sigma_*$. 
We chose the neighborhood $V$ so that
Lemma~\ref{Lemma:GoodNbhd} applies, hence such an   $A\in G$ is unique. From uniqueness (of $A$ and $A^2$), it follows easily that,
if $\zeta\in\Gamma\setminus\Gamma_+$ is the element such that $\sigma_*$ is conjugation by $\zeta$
then, by choosing $\rho(\zeta)=A$,
$\rho_+$ extends to $\rho\colon\Gamma\to G$. Hence:
$$
\mathrm{res}(\mathcal{R}(\Gamma, G))= \{[\rho]\in  \mathcal{R}(\Gamma_+, G) \mid \sigma^*([\rho])=[\rho]\}  .
$$
Let $U=\mathrm{res}^{-1}(V)$. With this choice of $U$ and~$V$,
$$ \mathrm{res}\colon U\to  \{[\rho]\in V\mid \sigma^*([\rho])=[\rho]\}$$ is a continuous bijection.

Finally we establish continuity of  $\mathrm{res}^{-1}$ using a slice. 
The existence of a slice $S\subset \mathcal{R}(\Gamma_+, G) $ at 
$\mathrm{res}(\mathrm{hol})$ is proved  by Johnson and Millson in 
\cite[Theorem~1.2]{JohnsonMillson}, who point to Borel-Wallach \cite[IX.5.3]{BorelWallach} for a definition
of slice.
From the properties of the slice, and as the stabilizer of $\mathrm{hol}\vert_{\Gamma_+}$ is trivial,
the natural map $G\times S\to  \mathcal{R}(\Gamma_+, G)$, that maps $(g, s)\in  G \times S$
to $gsg^{-1}$, yields a homeomorphism 
between $G\times S$ and  a neighborhood
of the orbit of $\mathrm{res}(\mathrm{hol})$, 
and the
projection induces  a homeomorphism $S\cong V$. It follows from the product structure that the
$A\in G$ that conjugates $\rho_+$ and $\rho_+\circ \sigma_*$ is continuous on $\rho_+$,
so the extension of $\rho_+$ to a representation of the whole $\Gamma$ is continuous on $\rho_+$.
Then  continuity of $\mathrm{res}^{-1}$ follows by composing the homeomorphism $V\cong S$
(restricted to the fixed point set of $\sigma^*$) with the extension from $\Gamma_+$  to $\Gamma$, and projecting to 
 $U\subset \mathcal{R}(\Gamma, G)$.
\end{proof}

As $\Gamma_+$ preserves the orientation, next we use the complex structure of the 
identity component 
$G_0= \operatorname{Isom}^+(\mathbb{H}^3)\cong \mathrm{PSL}(2,\mathbb C)$.

\subsection{Representations  in $ \mathrm{PSL}(2,\mathbb C)$}

The holonomy of the orientation covering $M_+^3$ is contained in $ \mathrm{PSL}(2,\mathbb C)$,
and it is well defined up to the action of $G=\mathrm{PSL}(2,\mathbb C)\rtimes\mathbb Z_ 2$
by conjugation. If we furthermore choose an orientation on $M_+^3$, then the holonomy is unique up
to the action by conjugacy of $G_0= \operatorname{Isom}^+(\mathbb{H}^3)\cong \mathrm{PSL}(2,\mathbb C)$,
and complex conjugation corresponds to changing the orientation.
We call the conjugacy class  in $\mathrm{PSL}(2,\mathbb C)$ of the holonomy of 
$M^3_+$ the \emph{oriented holonomy}.

 We consider
$$
 \mathcal{R}(\Gamma_+, \mathrm{PSL}(2,\mathbb C))=
 \hom( \Gamma_+,  \mathrm{PSL}(2,\mathbb C))/ \mathrm{PSL}(2,\mathbb C).
$$
Its local structure is well known:

\begin{Theorem}
\label{Thm:smooth}
A neighborhood of the oriented holonomy of $M_+^3$ in $\mathcal{R}(\Gamma_+, \mathrm{PSL}(2,\mathbb C))$
has a natural structure of $\mathbb{C}$-analytic variety
defined over $\mathbb{R}$. 

\end{Theorem}

The fact that it is $\mathbb C$-analytic follows for instance from~\cite{JohnsonMillson} or~\cite{KapovichBook}.
In Theorem~\ref{Thm:coordinates} we precise  $\mathbb C$-analytic coordinates, for the moment this is sufficient for our purposes.

\begin{Lemma}
\label{Lemma:notfixed}
 Let $\mathrm{hol}_+$ be the oriented holonomy of $M_+^3$. Then
 $$
 [\mathrm{hol}_+]\neq [\overline{\mathrm{hol}_+}]\in \mathcal{R}(\Gamma_+, \mathrm{PSL}(2,\mathbb C)).
 $$
 Namely, the  oriented holonomy and its complex conjugate are not conjugate by a matrix in 
 $\mathrm{PSL}(2,\mathbb C)$.
\end{Lemma}

\begin{proof}
By contradiction, assume that $\mathrm{hol}_+$ and $\overline{\mathrm{hol}_+}$ are 
conjugate by a matrix in 
$\mathrm{PSL}(2,\mathbb C)$: there exists an orientation-preserving isometry $A\in \mathrm{PSL}(2,\mathbb C)$ such that
$$
A\,\mathrm{hol}_+(\gamma)\, A^{-1}= \overline{\mathrm{hol}_+(\gamma)},\qquad\forall\gamma\in \Gamma_+.
$$
Consider the  orientation reversing isometry $B=c\circ A$, where $c$ is the isometry with  M\"obius transformation the complex conjugation, $z \mapsto \overline{z}$. 
The previous equation is equivalent to
\begin{equation}
 \label{eqn:commute}
 B\,\mathrm{hol}_+(\gamma)\, B^{-1}={\mathrm{hol}_+(\gamma)},\qquad\forall\gamma\in \Gamma_+.
\end{equation}
Brouwer's fixed point theorem yields that the fixed point set of $B$ in the ball compactification $\mathbb{H}^3\cup\partial_{\infty} \mathbb{H}^3$
is non-empty: 
$$
\mathrm{Fix}(B)=\{x\in \mathbb{H}^3\cup\partial_{\infty} \mathbb{H}^3\mid B(x)=x\}\neq\emptyset.
$$
By \eqref{eqn:commute}  $\mathrm{hol}_+(\Gamma_+)$  preserves $\mathrm{Fix}(B)$. Thus, by minimality of the limit set of a Kleinian group, since 
$\mathrm{Fix}(B)\neq\emptyset$ is closed and  $\mathrm{hol}_+(\Gamma_+)$-invariant, it contains the whole ideal boundary: 
$\partial_{\infty} \mathbb{H}^3\subset \mathrm{Fix}(B)$. Hence $B$ is the identity, contradicting that $B$ reverses the  orientation.
\end{proof}

From Lemma~\ref{Lemma:notfixed} and Theorem~\ref{Thm:smooth}, we have:

\begin{Corollary}
 There exists a neighborhood $W\subset \mathcal{R}(\Gamma_+, \mathrm{PSL}(2,\mathbb C))$
 of the conjugacy class  of the oriented holonomy of $M_+$ that is disjoint from its complex conjugate:
 $$
 \overline{W}\cap W=\emptyset.
 $$
\end{Corollary}

By choosing the neighborhood $W\subset \mathcal{R}(\Gamma_+, \mathrm{PSL}(2,\mathbb C))$
sufficiently small, we may assume that its projection to 
$\mathcal{R}(\Gamma_+, G)$
is in contained in $V$ as in Lemma~\ref{Lemma:UV}.
The neighborhood $V$ can also be chosen smaller, to be equal to the projection of $W$, as this map is open. Namely the neighborhoods
can be chosen so that 
$   \mathcal{R}(\Gamma_+, \mathrm{PSL}(2,\mathbb C))\to \mathcal{R}(\Gamma_+, G)$
restricts to a homeomorphism between $W$ (or $\overline{W}$) and $V$. In particular we can lift to $W$ the
restriction map from $U$ to $V$: 
$$
\xymatrix{
                      &  & W \ar[d]^{\cong}\\
    U \ar[rr]^{\mathrm{res}}   \ar@{.>}[urr]^{\widetilde{\mathrm{res}}}   &   & V }
 $$

\begin{Lemma}
\label{Lemma:liftres}
For $U\subset \mathcal{R}(\Gamma, G)$ and $W \subset \mathcal{R}(\Gamma_+, 
 \mathrm{PSL}(2,\mathbb C))$ as above, the lift of the restriction map yields an homeomorphism:
  $$
  {\widetilde{\mathrm{res}}}\colon U\overset{\cong}\longrightarrow 
  \{[\rho]\in W\mid [\rho\circ\sigma_*]=[\overline{\rho}] \}.
  $$
\end{Lemma}

This lemma has same proof as Lemma~\ref{Lemma:UV}, just taking into account that 
$\rho(\zeta)\in G$ reverses the orientation, for
 $[\rho]\in U$ and $\zeta\in \Gamma\setminus\Gamma_+$.

\subsection{Local coordinates}
 
Here we give the local coordinates of Theorem~\ref{Thm:smooth} and we prove a stronger version of Theorem~\ref{Thm:dimdef}.

For $\gamma\in\Gamma_+$ and $[\rho]\in\mathcal{R}(\Gamma_+,\mathrm{PSL}(2,\mathbb{C}))$,
as Culler and Shalen in \cite{CullerShalen} define
\begin{equation}
 \label{eqn:Igamma}
I_{\gamma}([\rho])=(\mathrm{trace}(\rho(\gamma)))^2-4.
 \end{equation}
Thus $I_\gamma$ is a function from  $\mathcal{R}(\Gamma_+,\mathrm{PSL}(2,\mathbb{C}))$
to $\mathbb{C}$. This function plays a role in the  
generalization of  Theorem~\ref{Thm:dimdef}. 

\begin{Theorem}
 \label{Thm:coordinates}
Let $M_+^3$ be as above and assume that it has $n$ cusps. 
Chose $\gamma_1,\ldots,\gamma_n\in\Gamma_+$ a non-trivial element for each peripheral
subgroup. Then, for a neighborhood $W\subset \mathcal{R}(\Gamma_+,\mathrm{PSL}(2,\mathbb{C}))$
of the oriented holonomy,
$$
(I_{\gamma_1},\dots, I_{\gamma_n})\colon W\to\mathbb{C}^n
$$
 defines a bi-analytic map between $W$ and a neighborhood of the origin.
\end{Theorem}

This theorem holds for any orientable hyperbolic manifold of finite volume, though we only use it for the orientation covering. 
Again, see \cite{BoileauPorti,KapovichBook} for a proof. As explained in these references, this is the algebraic part 
of the proof of Thurston's hyperbolic Dehn filling theorem using varieties of representations.

For a Klein bottle $K^2$, in Definition~\ref{def:distiguished} 
we considered the presentation of its fundamental group:
$$
\pi_1(K^2)=\langle a,b \mid aba^{-1}=b^{-1}\rangle
.$$
The elements
$a^2$ and $b$ are called \emph{distinguished} elements. 
Recall that, in terms of paths, those are represented by the unique homotopy classes of  loops
in the orientation covering that are invariant by the deck transformation (as unoriented curves).

Here we prove the following generalization of Theorem~\ref{Thm:dimdef}:

\begin{Theorem}
 \label{Thm:coordinatesNO}
 Let $M^3$ be a non-orientable manifold of finite volume with $k$ non-orientable cusps 
 and $l$ orientable cusps.
 For each horospherical Klein bottle, $K^2_i$, chose $
 \gamma_i\in\pi_1(K_i^2)$ distinguished, $i=1,\ldots, k$. 
  For each horospherical torus, $T^2_j$, chose a nontrivial 
 $\mu_j\in\pi_1(T_j^2)$, $j=1,\ldots, l$.
 
 There exists a neighborhood $U\subset \mathcal{R}(\Gamma, G)$ 
 of the holonomy of $M^3$
 such that 
 the map
 $$
( I_{\gamma_1},\ldots, I_{\gamma_k},I_{\mu_1},\ldots,I_{\mu_l} )\circ \widetilde{\mathrm{res}}:
U   \to 
\mathbb{R}^k\times \mathbb{C}^{l}
 $$
 defines a homeomorphism between $U$ and a neighborhood of the origin in 
$\mathbb{R}^k\times \mathbb{C}^{l}$.
\end{Theorem}

\begin{proof}
Let $M^3_+\to M^3$ be the orientation covering. By construction, by the choice of distinguished
elements in the peripheral Klein bottles,  $\gamma_i\in \Gamma_+$. Furthermore, as the peripheral
tori are orientable, $\mu_j\in \Gamma_+$.
Hence 
$$
\{\gamma_1,\ldots,\gamma_k,\mu_1,\ldots,\mu_l,\sigma_*(\mu_1), \ldots, \sigma_*(\mu_l)\}
$$
gives a nontrivial element for each peripheral subgroup of $\Gamma_+$, where $\sigma_*$ is the group automorphism from Definition~\ref{dfn:sigma}
. We apply Theorem~\ref{Thm:coordinates}:
$$
I=(I_{\gamma_1}, \ldots, I_{\gamma_k},I_{\mu_1},\ldots,I_{\mu_l},
I_{\sigma_*(\mu_1)}, \ldots, I_{\sigma_*(\mu_l)})\colon W\to \mathbb C^{k+2l}
$$
is a bi-analytic map with a neighborhood of the origin.
Furthermore, as $\sigma
_*(\gamma_i)=\gamma_i^{\pm 1}$ and $(\sigma^*)^2=\operatorname{Id}$,
$$
I\circ\sigma^*\circ I^{-1}(x_1,\ldots,x_k,y_1,\ldots,y_l,z_1,\ldots, z_l)=
(x_1,\ldots,x_k,z_1,\ldots,z_l,y_1,\ldots, y_l).
$$
In addition, by construction $I$ commutes with complex conjugation. 
Hence, by Lemma~\ref{Lemma:liftres} the image
$
(I\circ\widetilde{\mathrm{res}})(U)
$
is the subset of a neighborhood of the origin in $\mathbb C^{k+2l}$ defined by
$$
\begin{cases}
x_i=\overline{x_i},& \forall i=1,\ldots,k,\textrm{ and } \\
z_j=\overline{y_j},& \forall j=1,\ldots,l. 
\end{cases}
$$
Finally, by combining Theorem~\ref{Thm:coordinates} and Lemma~\ref{Lemma:liftres}, the map 
$I\circ\widetilde{\mathrm{res}}$ is a homeomorphism between $U$ and its image.  
 \end{proof}

We can now state the  generalization of 
Theorem~\ref{Thm:main_thm_one_cusp} to several cusps. 
Here $D(1)\subset\mathbb C$ denotes a disk of radius~$1$.

\begin{theorem}
\label{Thm:main_thm_several_cusps}
        Let $M^3$ be a complete non-orientable hyperbolic $3$-manifold of finite
volume with $k$ non-orientable cusps and $l$ orientable cusps.
\begin{enumerate}[(a)]
\item If $M^3$ admits a geometric ideal triangulation
$\Delta$, then $\mathrm{Def}(M^3, \Delta)\cong (-1,1)^k\times D(1)^l$. 
The  parameters $(\pm t_1,\ldots, \pm l_k,\pm u_1,\ldots,\pm u_l) \in (-1,1)^k\times D(1)^l$ 
correspond to the same structure.

\item A neighborhood of the holonomy in $\mathcal{R}(\pi_1(M^3),\mathrm{Isom}(\mathbb{H}^3))$
is homeomorphic to $(-1,1)^k\times D(1)^l$.
\end{enumerate}
Furthermore,   the holonomy map
$\mathrm{Def}(M^3, \Delta) \to  \mathcal{R}(\pi_1(M^3),\mathrm{Isom}(\mathbb{H}^3))$
in coordinates writes as:
$$
\begin{array}{rcl}
 (-1,1)^k\times D(1)^l & \to &  (-1,1)^k\times D(1)^l
\\
( t_1,\ldots,  t_k, v_1,\ldots, v_l) & \mapsto &
( t_1^2,\ldots,  {t_k}^2, 
{v_1}^2,\ldots, {v_l}^2)
\end{array}
$$  
Namely, each interval $ (-1,1)$ is folded along $0$ and has image $[0,1)$, and disks $D(1)$ are mapped to disks by a  2:1 branched covering.
\end{theorem}

\begin{proof}
  Assertion (a) is
Theorem~\ref{deformation}, and assertion (b) is Theorem~\ref{Thm:coordinatesNO}. 
To describe the holonomy map in coordinates, for each cusp (orientable or not) choose  
an orientation preserving peripheral element $m$ 
and let $v$ be the logarithm of the holonomy of $m$ defined as 
in \eqref{dehn_equation} in a neighborhood of the origin in $\mathbb C$ (with $v\in i\mathbb{R}$ in the non-orientable case). In particular $v$ 
is a component of the local
coordinates of $\mathrm{Def}(M^3, \Delta)$. Furthermore, 
the holonomy of $m$
is conjugate to 
$$
\pm   \begin{pmatrix}
       e^{v/2} & 1 \\ 0 & e^{-v/2}  
        \end{pmatrix}
$$
therefore  it has trace $\pm 2\cosh\frac{v}{2}$, and this trace
is a component of the local
coordinates of  $\mathcal{R}(\pi_1(M^3),\mathrm{Isom}(\mathbb{H}^3))$. 
 Then the assertion follows from 
applying a suitable coordinate change.
\end{proof}

\section{Representations of the Klein bottle}
\label{S:RepKl}

Let $\pi_1(K^2)=\langle a,b | aba^{-1}=b^{-1} \rangle$ be a presentation of the fundamental group of the Klein bottle, and $G=\mathrm{Isom}(\mathbb{H}^3)\cong \mathrm{PSL}(2,\mathbb{C})\rtimes \mathbb{Z}_2$. The variety of representations $\mathrm{hom}(\pi_1(K^2),G)$ is identified with 
$$
\mathrm{hom}(\pi_1(K^2),G)\cong \{A,B \in G | ABA^{-1}=B^{-1}  \}.
$$
Topologically, we can expect to have, at least, $4$ (possibly empty) connected components according to the orientable nature of $A$ and $B$. We are interested in studying one of them.

\begin{Definition}
A representation $\rho \in \mathrm{hom}(\pi_1(K^2),G)$ is said to \emph{preserve the orientation type} if, for every $\gamma\in \pi_1(K^2)$, $\rho(\gamma)$ is an orientation-preserving isometry if and only if $\gamma$ is represented by and orientation-preserving loop of $K^2$. We denote this subspace of representations by
$$
\hom_+(\pi_1(K^2), G).
$$
\end{Definition}

Let $T^2 \rightarrow K^2$ be the orientation covering. The restriction map on the varieties of representations (without quotienting by conjugation) is:
$$
\mathrm{res}\colon \hom(\pi_1(K^2), G) \to \hom(\pi_1(T^2), \mathrm{PSL}(2,\mathbb{C}) ). 
$$

\begin{theorem}
\label{th:repklein}
Let $\rho\in \hom_+(\pi_1(K^2),G)$ preserve the orientation type and let $\rho(b)\neq \mathrm{Id}$. 
By writing $A=\rho(a)$, $B=\rho(b)$ as Möbius transformations, up to conjugation one of the following holds:
\begin{itemize}
	\item[a)] $A(z)=\overline z+1$, $B(z)=z+\tau i$, with $\tau\in \mathbb{R}_{> 0}$. 
	\item[a')] $A(z)=\overline z$, $B(z)=z+\tau i$,  with $\tau\in \mathbb{R}_{> 0}$. 
	\item[b)] $A(z)=e^l\overline{z}$, $B(z)=e^{\alpha i} z$, with  $l\in \mathbb{R}_{\geq 0}$, $\alpha \in (0,\pi]$. 
	\item[c)] $A(z)=e^{\alpha i}/\overline{z}$, $B(z)=e^l z$,  with  $l\in \mathbb{R}_{> 0}$, $\alpha \in [0,\pi]$. 
\end{itemize}
\end{theorem}

\begin{proof}
 Let $G^0=\mathrm{PSL}(2, \mathbb{C})\vartriangleleft G$ be the connected component of the identity. The variety of representations $\hom(\pi_1(T^2), G^0)/G^0$ is  well known. A representation $[\rho_0]$ in this variety is the class of either a parabolic representation, $\rho_0(l)(z)=z+1$, $\rho_0(m)(z)=z+\tau$, $\tau \in \mathbb{C}$, a parabolic degenerated one, $\rho_0(l)(z)=z$, $\rho_0(m)(z)=z+\tau$, $\tau \in \mathbb{C}$, or a hyperbolic one $\rho_0(l)(z)=\lambda z$, $\rho_0(m)(z)=\mu z$, $\lambda, \mu \in \mathbb{C}$, where $\pi_1(T^2)=\langle l, m | lm=ml \rangle$.
 
For $\rho_0 =\mathrm{res}(\rho)$, let $A= \rho(a)$, $B= \rho(b)$ where $a,b$ are generators of $\pi_1(K^2)$, 
and $L=\rho(l)$, $M=\rho(m)$. The following is satisfied:
 \begin{align}
 (A^2,B)= & (L,M), \tag{Restriction of a representation to the torus} \\
 ABA^{-1}= & B^{-1}. \tag{Klein bottle relation}
 \end{align}
In fact, in order for $\rho_0$ to be a restriction, there must be $A$ and $B$ satisfying the previous conditions. We prove the theorem using these equations.
 
If $[\rho]$ is in the parabolic case, by hypothesis $\tau\neq 0$. Then, the solution is unique and, $A(z)=\overline{z}+1$, $B(z)=z+\tau i$, $\tau \in \mathbb{R}\setminus \{0\}$, hence $L(z)=z+2$, $M(z)=z+\tau i$. Similarly, for the degenerated parabolic case, $A(z)=\overline{z}$, $B(z)=z+\tau i$, $\tau \in  \mathbb{R}\setminus \{0\}$.

On the other hand, for $[\rho]$ hyperbolic, either $L$ corresponds to a real dilation and $M$ to a rotation, or the other way around.  In the case $L(z)=e^{2l}z$, $M(z)=e^{\alpha i}z$, $l\in \mathbb{R}$, $\alpha \in (-\pi, \pi]$ the representation can be written as the restriction of several representations of the Klein bottle, but all of them are conjugated to $A(z)=e^{l} \overline{z}$, $B(z)=e^{\alpha i}z$. A similar situation happens when $L(z)=e^{2\alpha i}z$, $M(z)=e^{l}z$, $ l \in \mathbb{R}$, $\alpha \in (-\pi, \pi]$, obtaining $A(z)=e^{\alpha i}/\overline{z},  B(z)=e^lz$. However, in the last case, we should note down that for every such representation $[\rho]$, we get two non-conjugated representations $[\rho_1], [\rho_2]$ such that $[\rho_0]=\mathrm{res}([\rho_1])=\mathrm{res}([\rho_2])$, them differing in $A_1(z)=e^{\alpha i}/\overline{z}$, $A_2(z)=e^{(\alpha+\pi)i}/\overline{z}=-e^{\alpha i}/\overline{z}$.

Thus, we obtain a classification of representations in $\hom(\pi_1(K^2),G)/G^0$. 
To get the classification quotienting by the whole group, $\hom(\pi_1(K^2),G)/G$, we only have to see how the complex conjugation $c$ acts by conjugation on each representation: In $a)$, $a')$, $c$ maps $z+\tau i \mapsto z-\tau i$; in $b)$, $e^{\alpha i}z \mapsto e^{-\alpha i}z$; and in $c)$, $e^{\alpha i}/\overline{z} \mapsto e^{-\alpha i}/\overline{z}$. The choice $\alpha >0, l>0$ in $b), c)$  is obtained by taking into account that $[\rho]=[\rho^{-1}]$.
\end{proof}

\begin{Definition}
\label{def:namereps} According to the different cases in Theorem~\ref{th:repklein}, a representation 
$\rho\in \hom_+(\pi_1(K^2),G)$ is called:
\begin{itemize}
 \item \emph{parabolic non-degenerate}  in case a) and \emph{parabolic degenerate} in case a'),
 \item \emph{type I} in case b), and 
 \item \emph{type II} in case c).
\end{itemize}
Furthermore, type I or II are called non-degenerate if $l\neq 0$ or $\alpha\neq 0$
respectively, and degenerate otherwise.
\end{Definition}

\begin{remark}
\label{rmk:non-deg}
The holonomy of a non-orientable cusp restricts to a representation of
the Klein bottle that preserves the orientation type and is parabolic non-degenerate.

Furthermore, deformations of this representation still preserve the orientation type
and are non-degenerate (possibly of type I or II), by continuity.
\end{remark}
 
 For $\gamma\in \pi_1(T^2) \vartriangleleft \pi_1(K^2)$,  recall from \eqref{eqn:Igamma} that
 $$
 \begin{array}{rcl}
  I_{\gamma}\colon \hom(\pi_1(K^2), G) & \to & \mathbb{C} \\
  \rho& \mapsto & (\operatorname{trace}_{\mathrm{PSL}(2,\mathbb{C})}(\rho(\gamma)))^2-4,
 \end{array}
 $$
where $\operatorname{trace}_{\mathrm{PSL}(2,\mathbb{C})}$ means trace as matrix in
 $\mathrm{PSL}(2,\mathbb{C})$.

\begin{lemma}
\label{Lemma:typeofrep}
Let $\rho\in\hom(\pi_1(K^2), G)$ preserve the orientation type and $\rho(b)\neq \mathrm{Id}$. Then:
\begin{itemize}
 \item If $\rho$ is parabolic, then $I_{\gamma}(\rho)=0$, $\forall \gamma\in \pi_1(T^2) $.
 \item If $\rho$ is of type I, then  $I_{ a^2}(\rho)\geq 0$ and $I_{b}(\rho)< 0$.
 \item If $\rho$ is of type II, then  $I_{ a^2}(\rho)\leq 0$ and $I_{b}(\rho)> 0$.
\end{itemize}

\begin{proof}
	It is a straightforward computation from Theorem~\ref{th:repklein}.
\end{proof}
\end{lemma}

\begin{corollary}\label{Coro:notrealized}
\begin{enumerate}[a)]
 \item The holonomy of a representation in $\mathrm{Def}(M,\Delta)$ is of type I
\item Representations in a neighborhood of $[\mathrm{hol}]$  in $\mathcal R(M^3,G)$ 
are or both, type I and II.
\item In particular, the holonomy map
$\mathrm{Def}(M,\Delta)\to \mathcal R(M^3,G)$
is not surjective in a neighborhood of the holonomy.
\end{enumerate}
\end{corollary}

\begin{proof}
Assertion a) follows from Remark~\ref{rmk:hol_lm} and Assertion b) from Theorem~\ref{Thm:coordinatesNO},
in both cases using Lemma~\ref{Lemma:typeofrep}.
\end{proof}

\section{Metric completion}
\label{S:Completion}

As we deform non-compact manifolds, the deformations into non-complete manifolds are not unique
(eg~one can consider proper open subset of a non-complete manifold).
We are not discussing the different issues related to this non-uniqueness, 
just the existence of a deformation into  a metric that can be complete as a conifold (see below).

The main result of this section is Theorem~\ref{Theorem:completion}. In the orientable case, the metric completion after 
deforming an orientable cusp is a singular space with a singularity called of  
\emph{Dehn type} (that include non-singular manifolds), 
see Hodgson's thesis \cite{Hodgson} and Boileau--Porti \cite[Appendix~B]{BoileauPorti}. 
In the non-orientable case, the singularity is more specific, a so called  conifold. 
 
\subsection{Conifolds and cylindrical coordinates}
\label{Subsection:conifolds}

A \emph{conifold} is a metric length space locally isometric to the metric cone of constant curvature on a spherical
conifold of dimension one less, see for instance \cite{BLP}.
When, as topological space, a conifold is homeomorphic to a manifold, it is called a \emph{cone manifold}, but in general it is only a pseudo-manifold.
In dimension 2 conifolds are also cone manifolds, but in dimension three there may be points with a neighborhood
homeomorphic to the cone on a projective plane $P^2$.

We are interested in three local models of singular spaces, that as conifolds are:
\begin{itemize}
 \item The hyperbolic cone over a round sphere $S^2$. This corresponds to a point with a non-singular hyperbolic metric.
 \item The hyperbolic cone over $S^2(\alpha,\alpha)$, the sphere with two cone points of angle $\alpha$, that  is
 the spherical suspension of a circle of perimeter $\alpha$. It corresponds to a singular axis of angle~$\alpha$.
 \item The hyperbolic cone over $P^2(\alpha)$, the projective plane with a cone point of angle $\alpha$. This is the 
 quotient of the previous one by a metric involution, which is the antipodal map on each concentric sphere. 
\end{itemize}

Next we describe metrically those local models, by using cylindrical coordinates in the hyperbolic space. 
These coordinates are defined from a geodesic line $g$ in $\mathbb H^3$, and we fix a point in the unit normal bundle to $g$,
ie~a vector $\vec u$ of norm $1$ and perpendicular to $g$.
Cylindrical coordinates give a diffeomorphism:
$$
\begin{array}{rcl}
 \mathbb{H}^3\setminus g & \overset\cong\longrightarrow & (0,+\infty)\times \mathbb{R}/2\pi\mathbb{Z}\times \mathbb{R} \\
   p & \longmapsto & (r,\theta, h)
\end{array}
$$
where $r$ is the distance between $g$ and $p$, 
$\theta$ is the angle parameter (the angle between the parallel transport of $\vec u$ and 
the tangent vector to the orthogonal geodesic from $g$  to $p$) and $h$ is  the arc parameter of $g$, 
the signed distance between the base point of $\vec u$
and the orthogonal projection from $p$ to $g$, Figure~\ref{Figure:Cylindrical}.

\begin{figure}[h]
\begin{center}
\begin{tikzpicture}[line join = round, line cap = round, scale=1]
\draw[thick] (0,0)--(0,3);
\draw (0,2)--(1.5,1.7);
 \draw[ ->] (0,2)--(.5,1.6);
 \draw[ ->] (0,1)--(.5,.6);
\draw[thin] (0.5,1.9)  arc [radius=.5, start angle = -30, end angle = -55]; 
\draw(1.5,1.7)[fill=black] circle(.03);
\draw (-.3, 2.7) node{$g$}; 
\draw (1.8, 1.7) node{$p$};
\draw (1, 2) node{$r$};
 \draw (0.65, 1.7) node{$_\theta$}; 
\draw (.65, 0.7) node{$\vec u$}; 
 \draw (-.15,1.5) node{$ \left\{ \rule{0pt}{5mm} \right. $};
 \draw (-.4,1.5) node{$h$}; 
\end{tikzpicture}
\end{center}
 \caption{Cylindrical coordinates.}
\label{Figure:Cylindrical}
\end{figure}
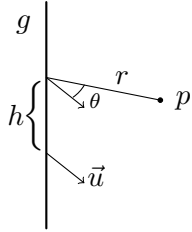

In the upper-half space model of $\mathbb H^3$, if $g$ is the geodesic from $0$ and $\infty$, then there exists a choice of coordinates 
(a choice of $\vec u$)
so that 
the projection from $g$ to the ideal boundary $\partial_\infty \mathbb{H}^3$ maps a point with
cylindrical coordinates $(r,\theta, h)$ to $e^{h+i\theta}\in\mathbb C$, Figure~\ref{Figure:Projection}.
A different choice of $\vec u$ would yield instead $\lambda e^{h+i\theta}\in\mathbb C$, for some $\lambda\in\mathbb C\setminus\{0\}$.

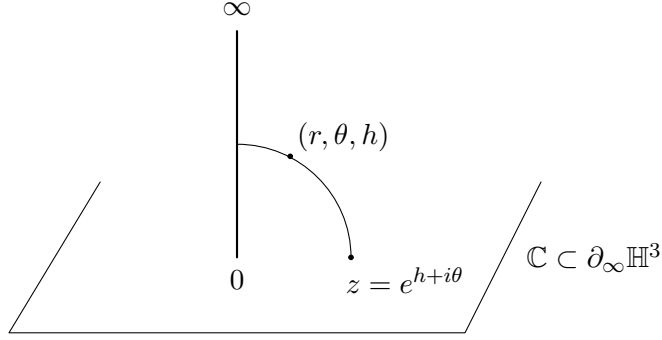
\begin{figure}[h]
\begin{center}
\begin{tikzpicture}[line join = round, line cap = round, scale=1]
\draw[thick] (0,0)--(0,3);
\draw (0,1.5) arc [radius=1.5, start angle = 90, end angle=0];
\draw(1.5,0)[fill=black] circle(.03);
\draw(.7,1.34)[fill=black] circle(.03);

\draw[thin] (-1.8,1)--(-3,-1)--(3,-1)--(4,1);

\draw (0,-.3) node{$0$ };
\draw (0,3.3) node{$\infty$};
\draw (2.21,-.3) node{$z=e^{h+i\theta}$ };
\draw(1.4,1.6) node{$(r,\theta,h)$ };
\draw(4.7,0) node{$ \mathbb{C}\subset \partial_\infty\mathbb{H}^3 $};
\end{tikzpicture}
\end{center}
 \caption{Orthogonal projection to $\partial_\infty\mathbb H^3$ with $g$ the geodesic with ideal end-points $0$ and $\infty$.}
\label{Figure:Projection}
\end{figure}

The hyperbolic metric on $\mathbb H^3$ in these coordinates is 
$$
d r^2+\sinh^2 (r) d\theta^2+\cosh^2 (r) d h^2
$$
More precisely, $\mathbb H^3$ is the metric completion of $(0,+\infty)\times \mathbb{R}/2\pi\mathbb{Z}\times \mathbb{R}$
with this metric.

\begin{Definition}
 For $\alpha\in (0,2\pi)$, $\mathbb{H}^3(\alpha)$ is the metric completion of 
 $(0,+\infty)\times \mathbb{R}/2\pi\mathbb{Z}\times \mathbb{R}$ for the metric
 $$
d s^2= d r^2+\left(\frac{\alpha}{2\pi }\right)^2 \sinh^2 (r) d\theta^2+\cosh^2 (r) d h^2
 $$
\end{Definition}

The metric space   $\mathbb{H}^3(\alpha)$  may be visualized by taking a sector in $\mathbb{H}^3$ of angle $\alpha$ and identifying its sides by a rotation. 
Alternatively, with the change of coordinates $\widetilde \theta = \frac{\alpha}{2\pi } \theta$, $\mathbb{H}^3(\alpha)$ is the metric completion of
$(0,+\infty)\times \mathbb{R}/\alpha \mathbb{Z}\times \mathbb{R}$ for the metric 
$
d r^2+\sinh^2 (r) d{\widetilde\theta}^2+\cosh^2 (r) d h^2
$.

\begin{Remark}
 The metric models are:
\begin{itemize}
 \item For the non-singular case (the cone on the round sphere) it is $\mathbb{H}^3$.
 \item For the singular axis (the cone on $S^2(\alpha,\alpha)$) it is $\mathbb{H}^3(\alpha)$.
 \item For the cone on $P^2(\alpha)$, it is the quotient
 $$
 \mathbb{H}^3(\alpha)/(r,\theta,h)\sim (r,-\theta,-h).
 $$
\end{itemize} 
\end{Remark}
 
 \subsection{Conifolds bounded by a Klein bottle}
 
We keep the notation of Subsection~\ref{Subsection:conifolds}, with cylindrical coordinates. Before discussing conifolds bounded by 
a Klein bottle, we describe a cone manifold bounded by a torus.

\begin{Definition}
A  \emph{solid torus with singular soul} is $ \mathbb{H}^3(\alpha)/\!\!\sim$,
where $\sim$ is the relation induced by the isometric action of $\mathbb{Z}$ generated by
$$
(r,\theta,h)\mapsto (r, \theta+\tau, h+L)
$$
for $\tau\in \mathbb{R}/2\pi\mathbb{Z}$ and $L>0$.
\end{Definition}
 
The space  $ \mathbb{H}^3(\alpha)/\!\!\sim$ is a solid torus of infinite radius with singular soul of cone angle $\alpha$, 
length of the singularity $L>0$ and torsion parameter $\tau\in \mathbb{R}/2\pi\mathbb Z$
(the rotation angle induced by parallel transport along the singular geodesic is 
$\frac{\alpha}{2\pi}\tau\in \mathbb{R}/\alpha\mathbb Z$).

By considering the metric neighborhood of radius $r_0>0$ on the singular soul, we get a compact solid torus,
bounded by a $2$-torus. This compact solid torus depicts a tubular neighborhood of a component of the singular 
locus of a cone manifold (compare Hodgson--Kerckhoff \cite{HodgsonKerckhoff} and
Hodgson's thesis \cite{Hodgson}).

We describe two conifolds bounded by a Klein bottle, that are a quotient of this solid torus by an involution.

\begin{Definition}
A  \emph{solid Klein bottle with singular soul} is $ \mathbb{H}^3(\alpha)/\!\!\sim$,
where $\sim$ is the relation induced by the isometric action of $\mathbb{Z}$ generated by
$$
(r,\theta,h)\mapsto (r, -\theta, h+L)
$$
for $L>0$.
\end{Definition}

 The space  $ \mathbb{H}^3(\alpha)/\!\!\sim$ is a solid Klein bottle of infinite radius with singular soul of cone angle~$\alpha$, 
and length of the singularity $L>0$. We may consider a metric tubular neighborhood
of radius $r_0$, bounded by a Klein bottle.
Its orientation cover is a solid torus with singular soul, cone angle $\alpha$, length of the singularity $2L$ and torsion parameter $\tau=0$.

\begin{Definition}
The   \emph{disc orbi-bundle with singular soul} is $ \mathbb{H}^3(\alpha)/\!\!\sim$,
where $\sim$ is the relation induced by two isometric involutions:
$$
\begin{array}{rcl}
(r,\theta,h)&\mapsto &(r, \theta+\pi, -h)\\
(r,\theta,h)&\mapsto &(r, \theta+\pi, 2L-h)
\end{array}
$$
for $L>0$.
\end{Definition}

To describe this space, it is useful first to look at the action on the preserved geodesic, corresponding to $r=0$. 
These involutions map $h\in\mathbb{R}$ to $-h$ and to $2L-h$ respectively. Thus it is the action of the infinite dihedral group
$\mathbb Z_2*\mathbb Z_2$ on a line generated by two reflections. Its orientation preserving subgroup is $\mathbb{Z}$
acting by translations on $\mathbb{R}$. Thus $\mathbb{R}/\mathbb{Z}$ is a circle, and 
$\mathbb{R}/(\mathbb Z_2*\mathbb Z_2)$ is an orbifold. The solid torus is a disc bundle over the circle, and our space is an orbifold-bundle
over $\mathbb{R}/(\mathbb Z_2*\mathbb Z_2)$ with fibre a disc. 

This space is the quotient of an involution on the solid torus. View the solid torus as two $3$-balls joined by two 1-handles,
Figure~\ref{Figure:SolidTorus}.
On each 3-ball apply the antipodal involution (on each concentric sphere of given radius), and extend this involution by permuting the 1-handles.
The quotient of each ball is the (topological) cone on $P^2$, hence our space is the result of joining two cones on $P^2$ by a $1$-handle. 
Its boundary is the connected sum
$P^2\# P^2\cong K^2$.

\begin{figure}[h]
\begin{center}
\begin{tikzpicture}[line join = round, line cap = round, scale=.7]

\begin{scope}[shift={(3,0.3)}, rotate=180]
\draw[fill=white] (0.5, 0) to[out=-15, in=180+15] (2.5, 0) to [out=-15, in=+15] (2.5,-.5) to[out=180+15, in=-15]  (0.5, -.5) to[out=180-15, in=180+15] 
cycle ;
\draw[fill=cyan!30!white, opacity=.3] (0.5, 0) to[out=-15, in=180+15] (2.5, 0) to [out=-15, in=+15] (2.5,-.5) to[out=180+15, in=-15]  (0.5, -.5) to[out=180-15, in=180+15]  cycle ; 
\end{scope}

 \draw[fill=white] (0,0) circle [radius=1];
 \draw[fill=white] (3,0) circle [radius=1];

\shadedraw[ white, fill=green!30!white, opacity=.3] (0,0) circle [radius=1];
 \draw (0,0) circle [radius=1];
\draw[very thin] (-1,0) arc [ radius=2, start angle=240, end angle= 300 ] ; 
\begin{scope}[shift={(3,0)}]
\shadedraw[ white, fill=green!30!white, opacity=.3] (0,0) circle [radius=1];
 \draw (0,0) circle [radius=1];
\draw[very thin] (-1,0) arc [ radius=2, start angle=240, end angle= 300 ] ;  
\end{scope}

\draw[fill=white] (0.5, 0) to[out=-15, in=180+15] (2.5, 0) to [out=-15, in=+15] (2.5,-.5) to[out=180+15, in=-15]  (0.5, -.5) to[out=180-15, in=180+15] 
cycle ;
\draw[fill=cyan!30!white, opacity=.4] (0.5, 0) to[out=-15, in=180+15] (2.5, 0) to [out=-15, in=+15] (2.5,-.5) to[out=180+15, in=-15] 
(0.5, -.5) to[out=180-15, in=180+15]  cycle ;
\end{tikzpicture}
\end{center}
\caption{A solid torus as  two $3$-balls joined by two $1$-handles.}
\label{Figure:SolidTorus}
\end{figure}
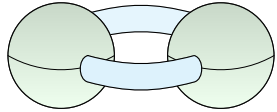

The singular locus of the  disc orbi-bundle $ \mathbb{H}^3(\alpha)/\!\!\sim$ is an interval (the underlying space  
of the orbifold bundle) of length $L$. 
The interior
points of the singular locus have cone angle $\alpha$, and the boundary points of the interval are precisely the points where it is not a
topological manifold.

Again $ \mathbb{H}^3(\alpha)/\!\!\sim$ has $\infty$ radius, and the metric tubular neighborhood of radius $r$ of the singularity
is bounded by a Klein bottle. It is the quotient of a solid torus of length $2L$ and torsion parameter $\tau=0$ by an isometric involution with two fixed points
(thus, as an orbifold, its orientation orbi-covering is a solid torus).

\begin{remark}  The boundary of both, the solid Klein bottle and the disc orbi-bundle,
is a Klein bottle. In both cases the holonomy preserves the orientation type, but the type of the presentation
as in Definition~\ref{def:namereps} is different:
\begin{enumerate}[a)]
 \item The holonomy of the boundary of a solid Klein bottle
 with singular soul is a representation of type I.
 
 \item The holonomy of the boundary of a disc orbi-bundle over a singular interval is of type II.
 \end{enumerate}
\end{remark}

For a non-orientable cusp, the holonomy of the peripheral torus is either parabolic non-degenerate, of type I or of type II, also nondegenerate (Remark~\ref{rmk:non-deg}).
The aim of next section is to prove that the deformations can be defined
so that the metric completion is either
solid Klein bottle with singular soul or a disc orbi-bundle with singular soul,
according to the type. This is the content of Theorem~\ref{Theorem:completion}, that we prove at the end of the section.

\subsection{The radial structure}

Let $M^3$ be a non-compact hyperbolic 3-manifold of finite volume. We deform its holonomy representation and
accordingly we deform its hyperbolic metric. Nonetheless,
incomplete metrics are not unique, so here we give a statement about the existence of a maximal structure, which corresponds to the one completed in Theorem~\ref{Theorem:completion}.

Let $[\rho] \in \mathcal{R}(\pi_1(M^3), G)$ be a deformation of its complete structure. 
There is some nuance in associating to $[\rho]$ a hyperbolic structure which is made explicit
by Canary--Epstein--Green in \cite{CEG}. Here, the authors conclude that every deformation with a given holonomy representation are related by an isotopy of the inclusion of $M^3$ in some fixed thickening $(M^{3})^*$, where a thickening is just another hyperbolic $3$-manifold containing ours.

We will start by making clear what we mean by a maximal structure.

\begin{Definition}
Let $M$ be a manifold with an analytic $(G,X)$-structure. 
We say that $M^*$ is an \emph{isotopic thickening of $M$} if it is a thickening and there is a isotopy, $i'$, of the inclusion, $i: M \hookrightarrow M^*$, such that $i'(M)=M^*$.
\end{Definition}

Given two isotopic thickenings of $M$  we say that $M_1^*\leq M_2^*$ if there is a $(G,X)$ isomorphism from $M_1^*$ to some subset of $M_2^*$ extending the identity on $M$. Hence, we say that an isotopic thickening is maximal if it is with respect the partial order relation we have just defined.

In general, it is not clear whether maximal isotopic thickenings exist, nor 
under which circumstances they do exist. However, we will construct in our situation an explicit maximal
thickening.

\begin{lemma}
	\label{lm:inj_rad}
	Let $\mathrm{inj}_{M^3}(x)$ denote the injectivity radius at a point $x\in M^3$. Then, a necessary condition for a non-trivial thickening of $M^3$ to exist is that there must exist a sequence $\{x_n\}\subset M^3$ with $\mathrm{inj}_{M^3}(x_n)\rightarrow 0$.
\end{lemma}

\begin{proof}
	Let us suppose a thickening $(M^3)^*$ exists. Then, take a point $x\in \partial((M^3)^*\setminus M^3)$. Any sequence $\{x_n\}\subset M^3$ such that $x_n\rightarrow x$ satisfies $\mathrm{inj}_{M^3}(x_n)\rightarrow 0$.
\end{proof}

The purpose of Lemma~\ref{lm:inj_rad}  is two-fold: first, it gives a condition for a thickening to be maximal (in the sense of the partial order relation we just defined), and second, it shows where a manifold could possibly be thickened. Taking into account a thick-thin decomposition of the manifold, the thickening can only be done in the deformed cusps.

Each cusp of $M^3$ is diffeomorphic to either $T^2\times [0,\infty)$ or $K^2\times [0,\infty)$. 
Let us consider a proper product compact subset $K^2\times[0,\lambda]$ or $T^2\times [0,\lambda]$ of an end, for some $\lambda>0$, and let us denote by $D_\rho$ the developing map of a structure with holonomy $\rho$ in the equivalence class $[\rho]\in \mathcal{R}(\Gamma,G)$.

\begin{lemma}
	\label{lm:section_developing}
	The image of the proper product subset by the developing map, $D_{\rho}(\widetilde{K^2}\times[0,\lambda])$, $D_{\rho}(\widetilde{T^2}\times[0,\lambda])$ lies within two tubular neighborhoods of a geodesic $\gamma \in M^3$, that is, in $N_{\epsilon_2}(\gamma)\setminus N_{\epsilon_1}(\gamma)$, where $N_\epsilon(\gamma)=\{x\in \mathbb{H}^3\mid d(x,\gamma)<\epsilon \}$. Moreover, for every geodesic ray exiting orthogonally from $\gamma$, the intersection of the ray with $D_{\rho}(\widetilde{K^2}\times[0,\lambda])$ is non-empty and transverse to any section $D_{\rho}(\widetilde{K^2}\times\{\mu\})$, $\mu\in[0,\lambda]$ and, analogously for an orientable end.
\end{lemma}

\begin{proof}
	We use a modified argument of Thurston (see his notes 
	\cite[Chapters 4 and 5]{ThurstonNotes}) to prove the lemma for a non-orientable end (the same idea goes for an orientable one). The original argument of Thurston shows that in an ideal triangulated manifold, the image of the universal cover of the end under the developing map is the whole tubular neighborhood but the geodesic. Let $[\rho_0]$ be the parabolic representation corresponding to the complete structure, then $D_{\rho_0}(\widetilde{K^2}\times[0,\lambda])$ is the region between two horospheres centered at an ideal point  $p_\infty \in \partial_{\infty}\mathbb{H}^3$. Let $K\subset \widetilde{K^2}\times[0,\lambda]$ denote a fundamental domain of $K^2\times[0,\lambda]$. The domain $K$ can be taken so that $D_{\rho_0}(\overline{K})$  is a rectangular prism between two horospheres. 
	
We want to deform 
 $D_{\rho_0}(\overline{K})$  as we deform $\rho_0$ to
$\rho$. We do that by deforming the horosphere centered at $p_\infty$ 
to surfaces equidistant to the geodesics $\gamma_\rho$ invariant by the holonomy of
the peripheral subgroup $\rho(\pi_1(K^2))$.
The deformation of the horosphere to equidistant surfaces is described in 
\cite[\S4.4]{ThurstonNotes} in the half-space model of $\mathbb H^3$, see also
Benedetti--Petronio \cite[\S E.6.iv]{BenedettiPetronio}.
Alternatively, we can view the deformation of the horosphere  to the  equidistant surfaces as follows.
Consider $\mathbb Z^2<\pi_1(K^2)$ the orientation preserving subgroup of index $2$,
$\rho(\mathbb Z^2)$ is contained in a unique one-complex parameter subgroup  $U_\rho\subset \mathrm{PSL }(2,\mathbb C)$
 (ie~$U_\rho$ is the exponential image of a $\mathbb C$-line
in the Lie algebra $\mathfrak{sl}(2,\mathbb C)$). This $U_\rho$ depends continuously on $\rho$,
 and given $x\in\mathbb H^3$
 the orbit $U_\rho (x)=\{g(x)\mid g\in U_\rho \}$ is a surface containing $x$ such that: 
 when $\rho=\rho_0$  then   $U_\rho (x)$  is 
 a horosphere centered at $p_\infty$, and when  when $\rho\neq \rho_0$, then  $U_\rho  (x)$ is a surface equidistant to the geodesic $\gamma_\rho$.
 Using this construction,  the image of the domain
$D_{\rho_0}(\overline{K})$ deforms to  $D_{\rho}(\overline{K})$ with the required properties, by following the equidistant surfaces for the factor $\widetilde{K^2}$ and the geodesics orthogonal 
to these surfaces for the factor 
$[0,\lambda]$.
\end{proof}

\begin{Definition}
The geodesic of Lemma~\ref{lm:section_developing} is called the \emph{soul} of the end.
\end{Definition}

\begin{remark}
	The face of the section of proper product subset the cusp $K^2\times[0,\lambda]$ or $T^2\times [0,\lambda]$ that is glued to the thick part of the manifold is the section of the cusp which is further away from the geodesic. Hence, we will only consider thickenings \textquotedblleft towards" the soul.
\end{remark}

\smallskip

Let $x$ be a point in a cusp of the manifold and consider its image under the developing $y=D_{\rho}(\tilde{x})$ of any lift $\tilde{x}$. There is only one geodesic segment in $\mathbb{H}^3$ such that $\gamma(0)=y$ and goes towards the soul orthogonally. In cylindrical coordinates, if $y=(r,\theta, h)$, the image of the geodesic consists of $\{(t,\theta, h)\mid t\in [0,r]\}$. Let us denote by $\gamma_x$ the corresponding geodesic in $M^3$.

\begin{theorem}
There exists a maximal thickening $M^*$ of a half-open product $M=K^2\times[0,\lambda)$ or $T^2\times [0,\mu)$. It is characterized by the following property: for every point $x\in M$, the geodesic $\gamma_x$ can be extended in $M^*$ so that $D_{\rho}(\tilde{\gamma_x})$ is the geodesic whose cylindrical coordinates with respect to the soul are $\{(t,\theta, h)\mid t\in (0,r]\}$.
\end{theorem}

\begin{proof}
Given a cusp section $S:=K^2$ or $T^2$, a product subset of the end, $K:=S\times[0,\lambda]$, a fixed fundamental domain $K_0$ of $K$ and a small neighborhood of $K_0$, $N(K_0)$, the set $T:=\{t\in \mathrm{Deck}(\tilde{K}/K)\mid tN(K_0)\cap N(K_0) \neq \emptyset \}$ is finite, where $\mathrm{Deck}(\tilde{K}/K)$ denotes the group of covering transformations of the universal cover. Hence, we can suppose that $D_{\rho\mid(T\overline{K_0})}$ is an embedding.

Let $\mathcal{U}$ be an open cover of $K$ by simply connected charts. For each $U$, take a lift $U_0\in \tilde{\mathcal{U}}$ such that $U_0\cap K_0 \neq \emptyset$ and consider $D_{\rho}(U_0)$. Given such a lift $U_0$, the other possible lift that could have non-empty intersection with $K_0$ are  $tU$, for $t\in T$. Furthermore, we can always assume that the chart $U$ coincides with the image of $U_0$ under the developing map, $D_{\rho}(U_0)$. Thus, we can identify 
	$$K\cong (\bigcup_{U\in \mathcal{U}} D_{\rho}(U_0))/\sim,$$
	 where the equivalence relation is by the action of $\mathrm{hol}(t)$, for $t \in T$.

Each $U\in \mathcal{U}$ can be thickened by identifying $U$ with $D_{\rho}(U_0)$ and considering, in cylindrical coordinates, the set of rays $R(U)=\{(t,\theta,h)\in \mathbb{H}^3\setminus \{soul\}\mid \exists (t_0, \theta, h)\in U, t<t_0\}$. Given two lifts of two thickened charts $R(U_1)$ and $R(U_2)$ with non-empty intersection with $K_0$, we glue them together in the points corresponding to $\mathrm{hol}(t)(R(U_1))\cap R(U_2)$, where $t\in T$. This defines a thickening of the cusp $K^*$. 

We have yet to show that it is isotopic to the original (half-open) product subset. Let us consider the section $S\times\{0\}$ of the cusp, the radial geodesics $\gamma_x$ for $x\in S\times\{0\}$ define a foliation of $K^*$ of finite length. Moreover, due to Lemma~$\ref{lm:section_developing}$, the foliation is transversal to $S\times\{0\}$. By re-parameterizing the foliation and considering its flow, we obtain a trivialization of the cusp, $K^*\cong S\times[0,\mu)$. Similarly, $K^*\setminus K$ is also a product. This let us construct an isotopy from $K^*$ to $K$.

This thickening clearly satisfies the property that $\gamma_x\subset{K^*}$ can be extended so that $D_{\rho}(\tilde{\gamma_x})=\{(t,\theta, h)\mid t\in (0,r] \}$. By taking geodesics $\gamma_x$ to geodesics through the developing map, it is clear our thickening can be mapped into every other thickening satisfying this property. Furthermore, if we consider the thickenings to be isotopic, we obtain an embedding.

Regarding the maximality, we will differentiate between an orientable end and a non-orientable one. The general idea will be the same one, for another isotopic thickening $(K)^{**}$ to include ours, the developing map should map some open set $V$ into a ball $W$ around a point $y_0$ in the soul, what will led to some kind of contradiction.

If $K$ is non-orientable, let us denote the distinguishable generators of $\pi_1(K^2)$, $a,b$, with $aba^{-1}=b^{-1}$. If $[\rho]$ is type I, $y_0$ is fixed by $\rho(b)$. Let $y\in W\setminus\{\mathrm{soul}\}$ and $x\in V$ be its preimage. $W$ is invariant by $\rho(b)$ and, in addition, both $x$ and $b\cdot x$ belong to $V$. Now take the geodesic $\gamma: I \mapsto \widetilde{(K)^{**}}$ from $x$ to $x_0$ which corresponds to the geodesic from $y$ to $y_0$. By equivariance and continuity, $x_0=\lim \gamma(t)=\lim b\gamma(t)=bx_0$. This contradicts $b$ being a covering transformation. If $[\rho]$ is type II, the previous argument with $a^2$ holds.

If $K$ is orientable, we will follow the same arguments leading to the completion of the cusp (for more details see, for instance, \cite{BenedettiPetronio}). the deformation $[\rho]$ is characterized in terms of its generalized Dehn filling coefficients $\pm(p,q)$. The case $p=0$ or $q=0$ are solved as in the non-orientable cusp, so we have the $2$ usual cases, $p/q\in \mathbb{Q}$ or $p/q \in \mathbb{I}$. For $p/q\in \mathbb{Q}$, there exists $k>0$ such that $k(p,q)\in \mathbb{Q}^2$ and $(kp)a+(kq)b$ is a trivial loop in the new thickening. If $p/q \in \mathbb{I}$, then $y_0$ is dense in $\{\mathrm{soul}\}\cap V$, which is a contradiction. 
\end{proof}

\begin{Definition}
We call the previous thickening, the \emph{radial thickening} of the cusp.
\end{Definition}

\begin{remark}
If the manifold $M^3$ admits an ideal triangulation, the canonical structure coming from the triangulation is precisely the radial thickening of the cusp.
\end{remark}

\begin{figure}
	\centering
	\def\svgwidth{0.45\linewidth}
	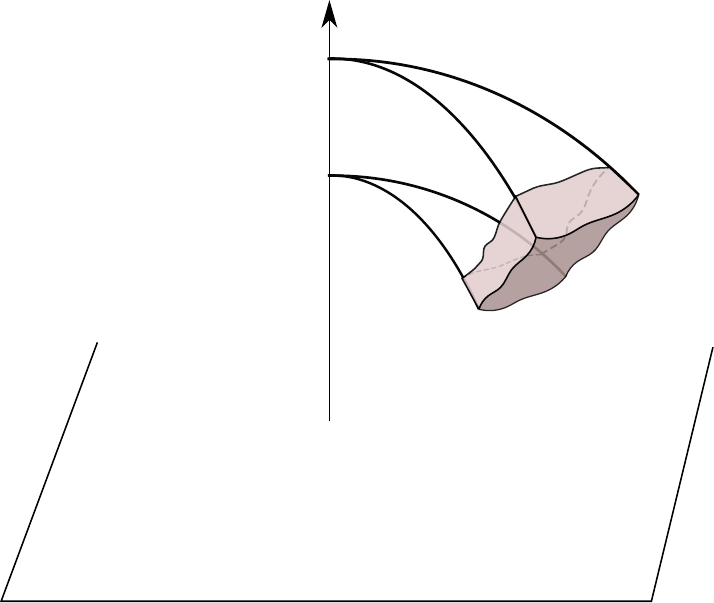
	\caption{The radial thickening.}
\end{figure}

\begin{Theorem} 
	\label{Theorem:completion}
	For a deformation of the holonomy $M^3$, the corresponding deformation of the metric can be chosen so that 
	on a non-orientable end: 
	\begin{itemize}
		\item It is a cusp, a metrically complete end, if the peripheral holonomy is parabolic.
		\item The metric completion is a solid Klein bottle with singular soul if the peripheral holonomy is of type I.
		\item The metric completion is a disc orbi-bundle with singular soul if it is of type II.
	\end{itemize}
	Furthermore, the cone angle $\alpha$ and the length $L$ of the singular locus is described by the peripheral boundary,
	so that those parameters start from $\alpha=L=0$ for the complete structure and grow continuously when deforming in either direction.
\end{Theorem}

\begin{proof}[Proof of Theorem~\ref{Theorem:completion}]
	The proof uses   the orientation covering and equivariance. More precisely, the deformation is constructed in the complete case for the 
	orientation covering and it can be made equivariant. The holonomy of a torus restricted from a Klein bottle
	is either parabolic or the holonomy of a solid torus with singular soul (and $\tau=0$). In particular the holonomy of a Klein bottle is
	parabolic iff its restriction to the orientable covering is parabolic. Furthermore, by using the description in cylindrical coordinates 
	(and using Figure~\ref{Figure:Projection}) and as $\tau=0$, the solid torus is equivariant by the action of 
	$\pi_2(K^2)/\pi_1(T^2) \cong \mathbb{Z}_2$.
\end{proof}

\section{Example: The Gieseking manifold}
\label{S:Gieseking}

We use the Gieseking manifold to illustrate the results of this paper. 
In particular the difference between deformation spaces obtained from ideal triangulations and from the variety of representations.

The Gieseking manifold $M$ is a non-orientable hyperbolic 3-manifold with finite volume and one cusp, with  
 horospherical section  a Klein bottle. It has an ideal triangulation
 with a single tetrahedron. The orientation cover of the Gieseking manifold is the figure eight knot exterior,
 and the ideal triangulation with one simplex lifts to Thurston's ideal triangulation with two ideal simplices in
 Thurston's notes \cite{ThurstonNotes}.

This manifold $M$ was constructed by Gieseking in his thesis in 1912, here we follow the description of Magnus in
\cite{Magnus}, using the notation of Alperin--Dicks--Porti \cite{ADP}.
Start with the regular ideal vertex $\Delta$ in $\mathbb H^3$, 
with vertices $\{0, 1, \infty, \frac{1-i\sqrt{3}}2\}$, Figure~\ref{fig:gieseking_labelled}. The side identifications are the non-orientable 
isometries defined by the M\"obius transformations
$$
 U(z)= \frac{ 1}{1+\tfrac{1+i\sqrt{3}}{2}\overline z} \qquad\textrm{ and }\qquad
 V(z)= -\tfrac{1+i\sqrt{3}}{2}\overline z+1 .
$$
The identifications of the faces are defined by their action on vertices:
$$
U\colon(\tfrac{1-i\sqrt{3}}2, 0 ,\infty)\mapsto (\tfrac{1-i\sqrt{3}}2,1,0)
\qquad\textrm{ and }\qquad
V\colon(1, 0 ,\infty)\mapsto (\tfrac{1-i\sqrt{3}}2,1,\infty).
$$

\begin{figure}[h]
\centering
\def\svgwidth{0.4\linewidth}
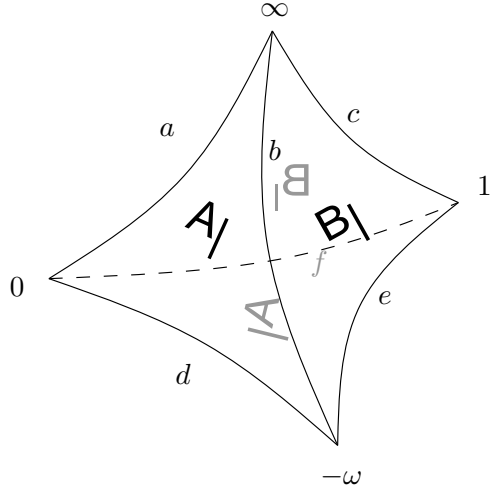
\caption{Gieseking Manifold with labeled edges.}
\label{fig:gieseking_labelled}
\end{figure}

By applying Poincar\'e's fundamental theorem
\begin{equation}
 \label{eqn:presentationG}
\pi_1(M)\cong\langle U, V\mid VU=U^2V^2\rangle 
\end{equation}

The relation $VU=U^2V^2$ corresponds to a cycle of length 6 around the edge.

\subsection{The deformation space $\mathrm{Def}(M, \Delta) $}

We compute the deformation space of the triangulation with a single tetrahedron, as in Section~\ref{S:CombinatorialDef}.

For any ideal tetrahedron in $\mathbb{H}_3$, we set its ideal vertices in $0, \, 1, \infty \, $and $-\omega$, where $\omega$ is in $\mathbb{C}_+$, the upper half-space of $\mathbb{C}$. The role played by $-\omega$ will be the one of $\frac{1-i\sqrt{3}}{2}$ in the complete structure. For any such $\omega$ it is possible to glue the faces of the tetrahedron in the same pattern as in the Gieseking manifold via two orientation-reversing hyperbolic isometries, which we will call likewise $U$ and $V$.

For the gluing to follow the same pattern, it must map $U: \; (-\omega,\, 0, \infty ) \rightarrow (-\omega, \, 1, \, 0)$ and $V: \; (1, \,0 , \, \infty) \rightarrow (-\omega, \, 1, \infty)$. The orientation-reversing isometries $U$ and $V$ satisfying this are:
$$
U(z)=\frac{1}{\frac{1+\omega}{|\omega|^2}\overline{z}+1}  \qquad\textrm{ and }\qquad V(z)=-(1+\omega)\overline{z}+1.$$

Although it is always possible to glue the faces in the same pattern as in the Gieseking manifold, not for all them the gluing will have a hyperbolic structure. 

\bigskip

Let us label the edges as in Figure~\ref{fig:gieseking_labelled}. For the topological manifold to be geometric, we only have to check that the pairing is proper (see \cite{Ratcliffe}). In this case, the only condition which we need to satisfy is that the isometry that goes through the only edge cycle is the identity. This is given by:

$$
a \overset{V}{\longrightarrow} c \overset{V}{\longrightarrow} b \overset{U}{\longrightarrow} d \overset{U}{\longrightarrow} e \overset{V^{-1}}{\longrightarrow} f \overset{U^{-1}}{\longrightarrow} a,
$$
and, therefore, we will have a hyperbolic structure if and only if $U^{-1}V^{-1}U^2V^2=\textrm{Id}$. Doing this computation, we obtain the equation

\begin{equation}
\label{eqn:proper}
 |\omega(1+\omega)|=1.
\end{equation}

Let us show that this equation matches the one obtained from 
Definition~\ref{dfn:DefSpace}. If we denote by $z(a)$ the edge invariant of $a$ and analogously for the rest of the edges, we have that the equation describing the deformation space of the manifold in terms of this triangulation is
$$\frac{z(a)z(b)z(e)}{\overline{z(c)z(d)z(f)}}=1.$$
Writing down all of the edge invariants in terms of $z(a)$ by means of the tetrahedron relations results in the equation 
\begin{gather}
\label{equation_gieseking}
\frac{z(a)^2\overline{z(a)}^2}{(1-z(a))(1-\overline{z(a)})}=\frac{|z(a)|^4}{|1-z(a)|^2}=1.
\end{gather} 
If we substitute $z(a)=-1/\omega$, we obtain
$$
\frac{1}{\omega \overline{\omega}(\omega+1)(\overline{\omega}+1)}=1,
$$
which is equivalent to Equation~\eqref{eqn:proper}.

\begin{Remark}
	The set $\{w\in\mathbb C\mid \vert w(1+w)\vert=1\}$ is homeomorphic to $S^1$, and the deformation space
	$\{w\in\mathbb C\mid \vert w(1+w)\vert=1  \textrm{ and }  \mathrm{Im}(w)>0\}$ is homeomorphic
	to an open interval, see Figure~\ref{Figure:AlgebraicDefSpace}.
\end{Remark}

We justify the remark and Figure~\ref{Figure:AlgebraicDefSpace}.
Firstly, to prove that set of algebraic solutions  is homeomorphic to a circle, we write the defining equation
$  \vert w(1+w)\vert=1$ as
$$
\Big\vert \Big(w+\frac{1}{2}\Big)^2-\frac{1}{4}\Big\vert= 1
$$
Thus $\big(w+\frac{1}{2}\big)^2$ lies in the circle of center $\frac{1}{4}$ and radius $1$.
As this circle separates $0$ from $\infty$, 
the equation defines a connected covering of  degree two 
of the circle.
Secondly, the set of algebraic solutions is invariant by the 
involutions $w\mapsto \overline{w}$ and $w\mapsto -1-w$ (hence symmetric with respect to the real line and the line defined
by real part equal to $-\frac{1}{2}$). Furthermore
it intersects the real line at $w=\frac{-1\pm \sqrt{5}}{2}$
and the line with real part $-\frac{1}{2}$ at $\frac{-1\pm i\sqrt{3}}{2}$.

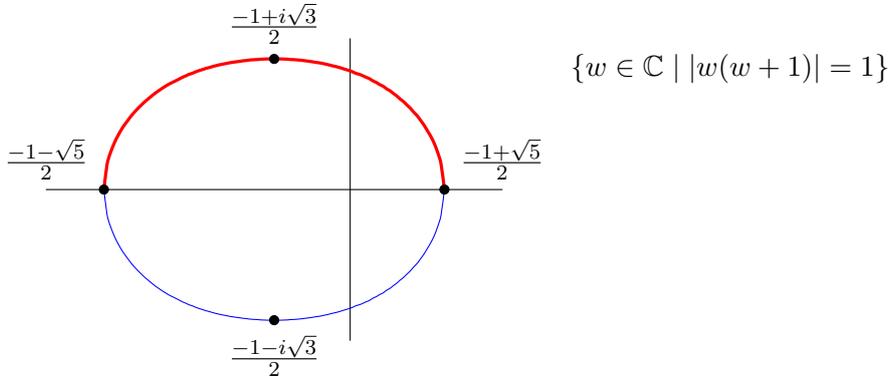
\begin{figure}[h]
	\begin{center}
		\begin{tikzpicture}[line join = round, line cap = round, scale=2]
		\draw[thin] (-2,0)--(1, 0);
		\draw[thin] (0,-1)--(0, 1);
		\begin{scope}[yscale=-1,xscale=1]
		\draw [blue, thin ] plot [smooth ] coordinates {
			(.620, 0.)  (.598, .169)  (.575, .248)  (.553, .306)  (.530, .353)  (.508, .390)  (.486, .427)  (.463, .459)  
			(.441, .489)  (.418, .513)  (.396, .539)  (.374, .560)  (.351, .583)  (.329, .600)  (.306, .621)  (.284, .637)  
			(.262, .655)  (.239, .668)  (.217, .685)  (.194, .696)  (.172, .710)  (.150, .720)  (.127, .735)  (.105, .744)  
			(0.082, .755)  (0.060, .764)  (0.038, .770)  (0.015, .781)  (-0.007, .789)  (-0.030, .796)  (-0.052, .803)  
			(-0.074, .810)  (-0.097, .816)  (-.119, .821)  (-.142, .827)  (-.164, .832)  (-.186, .837)  (-.209, .841)  (-.231, .844)  (-.254, .848)  
			(-.276, .851)  (-.298, .854)  (-.321, .857)  (-.343, .859)  (-.366, .861)  (-.390, .862)  (-.410, .864)  (-.430, .865)  (-.460, .866) 
			(-.480, .866)  (-.500, .866)  (-.520, .866)  (-.540, .866)  (-.570, .865)  (-.590, .864)  (-.610, .862)  (-.630, .861)  (-.660, .859)  
			(-.680, .856)  (-.700, .854)  (-.720, .851)  (-.750, .847)  (-.770, .844)  (-.790, .841)  (-.810, .837)  (-.840, .831)  (-.860, .827)  
			(-.880, .822)  (-.900, .817)  (-.930, .808)  (-.950, .803)  (-.970, .796)  (-.990, .790)  (-1.02, .779)  (-1.04, .770)  (-1.06, .764)  
			(-1.08, .754)  (-1.10, .746)  (-1.13, .730)  (-1.15, .720)  (-1.17, .710)  (-1.19, .700)  (-1.22, .682)  (-1.24, .669)  (-1.26, .655)  
			(-1.28, .641)  (-1.31, .617)  (-1.33, .600)  (-1.35, .583)  (-1.37, .564)  (-1.40, .534)  (-1.42, .512)  (-1.44, .489)  (-1.46, .462)  
			(-1.49, .421)  (-1.51, .389)  (-1.53, .353)  (-1.55, .312)  (-1.58, .236)  (-1.60, .163)  (-1.62, 0.)
		};
		\end{scope}
		\draw [red,  very thick ] plot [smooth ] coordinates {
			(.620, 0.)  (.598, .169)  (.575, .248)  (.553, .306)  (.530, .353)  (.508, .390)  (.486, .427)  (.463, .459)  
			(.441, .489)  (.418, .513)  (.396, .539)  (.374, .560)  (.351, .583)  (.329, .600)  (.306, .621)  (.284, .637)  
			(.262, .655)  (.239, .668)  (.217, .685)  (.194, .696)  (.172, .710)  (.150, .720)  (.127, .735)  (.105, .744)  
			(0.082, .755)  (0.060, .764)  (0.038, .770)  (0.015, .781)  (-0.007, .789)  (-0.030, .796)  (-0.052, .803)  
			(-0.074, .810)  (-0.097, .816)  (-.119, .821)  (-.142, .827)  (-.164, .832)  (-.186, .837)  (-.209, .841)  (-.231, .844)  (-.254, .848)  
			(-.276, .851)  (-.298, .854)  (-.321, .857)  (-.343, .859)  (-.366, .861)  (-.390, .862)  (-.410, .864)  (-.430, .865)  (-.460, .866) 
			(-.480, .866)  (-.500, .866)  (-.520, .866)  (-.540, .866)  (-.570, .865)  (-.590, .864)  (-.610, .862)  (-.630, .861)  (-.660, .859)  
			(-.680, .856)  (-.700, .854)  (-.720, .851)  (-.750, .847)  (-.770, .844)  (-.790, .841)  (-.810, .837)  (-.840, .831)  (-.860, .827)  
			(-.880, .822)  (-.900, .817)  (-.930, .808)  (-.950, .803)  (-.970, .796)  (-.990, .790)  (-1.02, .779)  (-1.04, .770)  (-1.06, .764)  
			(-1.08, .754)  (-1.10, .746)  (-1.13, .730)  (-1.15, .720)  (-1.17, .710)  (-1.19, .700)  (-1.22, .682)  (-1.24, .669)  (-1.26, .655)  
			(-1.28, .641)  (-1.31, .617)  (-1.33, .600)  (-1.35, .583)  (-1.37, .564)  (-1.40, .534)  (-1.42, .512)  (-1.44, .489)  (-1.46, .462)  
			(-1.49, .421)  (-1.51, .389)  (-1.53, .353)  (-1.55, .312)  (-1.58, .236)  (-1.60, .163)  (-1.62, 0.)
		};
		\draw(.62,0)[fill=black] circle(.03);
		\draw(-1.62,0)[fill=black] circle(.03);
		\draw(-.500, .866)[fill=black] circle(.03);
		\draw(-.500, -.866)[fill=black] circle(.03);
		\draw(1,.2) node{$\frac{-1+\sqrt{5}}{2}$};  
		\draw(-2,.2) node{$\frac{-1-\sqrt{5}}{2}$}; 
		\draw(-.5,1.1) node{$ \frac{-1+i\sqrt{3}}{2}$  };
		\draw(-.5,-1.1) node{$ \frac{-1-i\sqrt{3}}{2}$  };
		\draw(2.5,.8) node{$\{w\in\mathbb C\mid \vert w(w+1)\vert=1\}$};
		\end{tikzpicture}
	\end{center}
	\caption{ The set of solutions of the compatibility equations  and $\mathrm{Def}(M, \Delta)$ (the top half). }
	\label{Figure:AlgebraicDefSpace}
\end{figure}

Let us construct the link of the cusp. We denote the link of each cusp point as in Figure~\ref{fig:gieseking_manifold_link} and glue them to obtain the link as in Figure~\ref{fig:gieseking_link}, which is a Klein bottle.

\begin{figure}[h]
	\centering
	\begin{minipage}{.5\textwidth}
		\centering
		{\footnotesize
		\def\svgheight{0.2\textheight}
		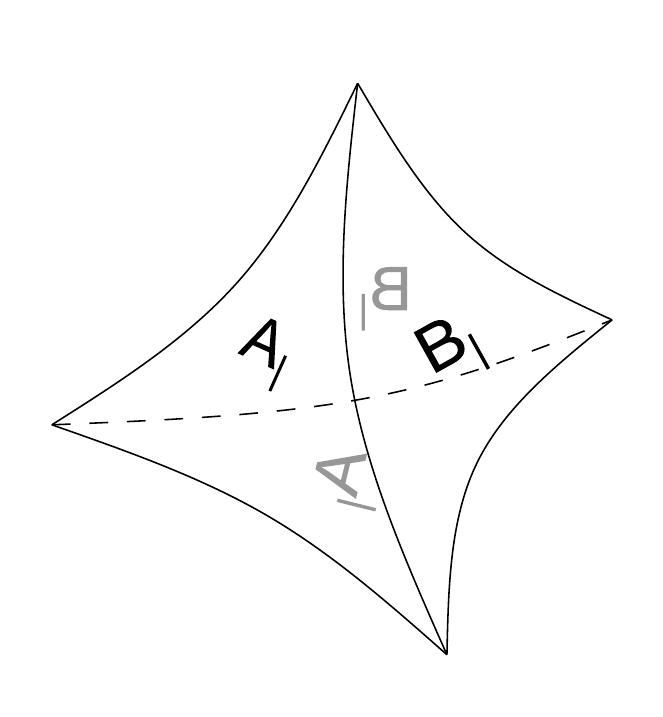}
		\caption{Gieseking manifold with link.}
		\label{fig:gieseking_manifold_link}
	\end{minipage}%
	\begin{minipage}{0.5\textwidth}
		\centering
		\def\svgheight{0.2\textheight}
\begingroup%
  \makeatletter%
  \providecommand\color[2][]{%
    \errmessage{(Inkscape) Color is used for the text in Inkscape, but the package 'color.sty' is not loaded}%
    \renewcommand\color[2][]{}%
  }%
  \providecommand\transparent[1]{%
    \errmessage{(Inkscape) Transparency is used (non-zero) for the text in Inkscape, but the package 'transparent.sty' is not loaded}%
    \renewcommand\transparent[1]{}%
  }%
  \providecommand\rotatebox[2]{#2}%
  \newcommand*\fsize{\dimexpr\f@size pt\relax}%
  \newcommand*\lineheight[1]{\fontsize{\fsize}{#1\fsize}\selectfont}%
    \ifx\svgwidth\undefined%
	\ifx\svgheight\undefined%
    	\setlength{\unitlength}{111.59984559bp}%
    	\ifx\svgscale\undefined%
      		\relax%
    	\else%
      		\setlength{\unitlength}{\unitlength * \real{\svgscale}}%
    	\fi%
    \else%
    	\def\svgratio{1.28046774}%
    	\setlength{\unitlength}{\svgheight / \real{\svgratio}}%
    \fi%
  \else%
    \setlength{\unitlength}{\svgwidth}%
  \fi%
  \global\let\svgwidth\undefined%
  \global\let\svgheight\undefined%
  \global\let\svgratio\undefined%
  \global\let\svgscale\undefined%
  \makeatother%
  \begin{picture}(1,1.28046774)%
    \lineheight{1}%
    \setlength\tabcolsep{0pt}%
    \put(0,0){\includegraphics[width=\unitlength,page=1]{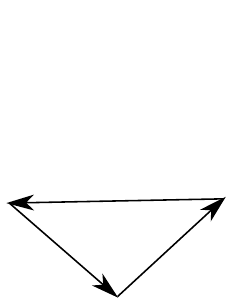}}%
    \put(0.31754025,0.15795718){\color[rgb]{0,0,0}\makebox(0,0)[lt]{\lineheight{0}\smash{\begin{tabular}[t]{l}$\zeta$\end{tabular}}}}%
    \put(0,0){\includegraphics[width=\unitlength,page=2]{giesekinglink.pdf}}%
    \put(0.0558092,0.59191517){\color[rgb]{0,0,0}\makebox(0,0)[lt]{\lineheight{0}\smash{\begin{tabular}[t]{l}$\alpha$\end{tabular}}}}%
    \put(0.84148442,0.60277703){\color[rgb]{0,0,0}\makebox(0,0)[lt]{\lineheight{0}\smash{\begin{tabular}[t]{l}$\alpha$\end{tabular}}}}%
    \put(0.4129483,0.41475319){\color[rgb]{0,0,0}\makebox(0,0)[lt]{\lineheight{0}\smash{\begin{tabular}[t]{l}$\gamma$\end{tabular}}}}%
    \put(0.54358033,0.6244518){\color[rgb]{0,0,0}\makebox(0,0)[lt]{\lineheight{0}\smash{\begin{tabular}[t]{l}$\beta$\end{tabular}}}}%
    \put(0.46857195,0.83483498){\color[rgb]{0,0,0}\makebox(0,0)[lt]{\lineheight{0}\smash{\begin{tabular}[t]{l}$\delta$\end{tabular}}}}%
    \put(0.2600026,0.98329445){\color[rgb]{0,0,0}\makebox(0,0)[lt]{\lineheight{0}\smash{\begin{tabular}[t]{l}$\epsilon$\end{tabular}}}}%
    \put(0.63145436,0.17997032){\color[rgb]{0,0,0}\makebox(0,0)[lt]{\lineheight{0}\smash{\begin{tabular}[t]{l}$\zeta$\end{tabular}}}}%
    \put(0.62930636,0.99053587){\color[rgb]{0,0,0}\makebox(0,0)[lt]{\lineheight{0}\smash{\begin{tabular}[t]{l}$\epsilon$\end{tabular}}}}%
  \end{picture}%
\endgroup%

		\caption{Link of the cusp point.}
		\label{fig:gieseking_link}
	\end{minipage}
\end{figure}

Now we take two tetrahedra and construct the orientation covering of $M$ (the figure eight knot exterior). For the first tetrahedron, we will denote by $z_1:=z(a)$, and $z_2, z_3$ so that they follow the cyclic order described in the tetrahedron relations. Afterwards, in the second tetrahedron, we denote by $w_i$ the edge invariant of the corresponding edge after applying an orientation reversing isometry to the tetrahedron, that is,  $w_i=\frac{1}{\overline{z_i}}$.

We consider the link of the orientation covering. The derivative of the holonomy of the two loops in the link of the orientation covering, $l_1$, $l_2$, depicted in Figure~\ref{fig:gieseking_longitude} (which are free homotopic)  is $\frac{w_1}{z_1}=\frac{1}{|w_1|^2}$ and $\frac{w_3}{z_3}=\frac{1}{|w_3|^2}$. For the manifold to be complete, $\mathrm{hol'}(l_i)=1$ for $i=1,2$, which happens if and only if $z_1=\frac{1}{2}+\frac{\sqrt{3}}{2}\textrm{i}$. This corresponds to the regular ideal tetrahedron, which, as expected, is the manifold originally given by Gieseking. Notice that the upper loop (the one going through the side $\epsilon$) can be taken as a \emph{distinguished} longitude. A suitable meridian is drawn in 
Figure~\ref{fig:gieseking_meridian}.

\begin{figure}[h]
	\centering
	\begin{minipage}{.5\textwidth}
		\centering
		\def\svgheight{0.2\textheight}
		{\scriptsize
		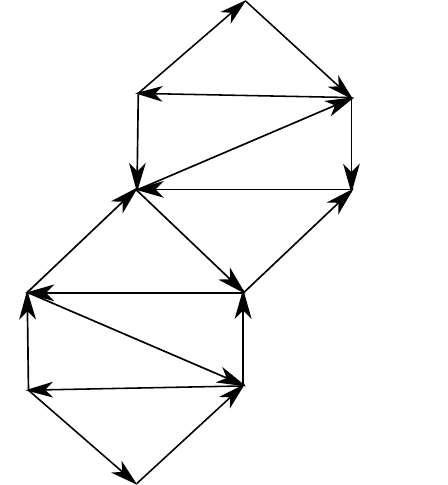}
\caption{Two free homotopic loops.}
\label{fig:gieseking_longitude}
	\end{minipage}%
	\begin{minipage}{0.5\textwidth}
		\centering
		\def\svgheight{0.2\textheight}
		{\scriptsize
		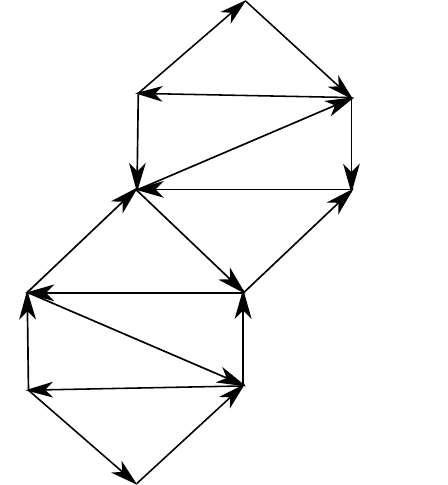}
\caption{Meridian in the link of the cover.}
\label{fig:gieseking_meridian}
	\end{minipage}
\end{figure}

Let us check that both the longitude and the meridian satisfy the conditions we stated for their holonomy in remark~\ref{rmk:hol_lm}, that is $\mathrm{hol'}(l)\in \mathbb{R}, |\mathrm{hol'}(m)|=1$. We have already shown it for the longitude. Regarding the meridian,
$$
\mathrm{hol'}(m)=\frac{z_2z_3w_2w_3}{w_2z_1z_2w_1}=\frac{z_1}{z_3}\frac{w_1}{w_3}=\frac{z_1}{\overline{z_1}}\frac{\overline{z_3}}{z_3},
$$
therefore $|\mathrm{hol}'(m)|=1$. This leads to the result that the generalized Dehn filling coefficients of a lifted structure have the form $(0,q)$, after an appropriate choice of longitude-meridian pair.

\smallskip

The last result could also had been obtained from Thurston's triangulation. By rotating the tetrahedra, our triangulation could be related with his, and the parameters identified. We can then check that in his choice of longitude and meridian, the holonomy has the same features if the structure is a lift from Gieseking manifold.

 \subsection{The Gieseking manifold as a punctured torus bundle}

The Gieseking manifold $M$ is  fibered over the circle with fibre a punctured torus $T^2\setminus\{*\}$. 
We use this structure to compute the variety of representations.
The monodromy of the fibration is an automorphism  
$$\phi\colon T^2\setminus\{*\}\to T^2\setminus\{*\}.$$
The map $\phi$ is the restriction 
of a map of the compact torus $T^2 \cong\mathbb R^2/\mathbb Z^2  $ that lifts 
 to the linear map of $\mathbb R^2$ with matrix 
$$\begin{pmatrix} 0 & 1 \\ 1 & 1
                                                                                           
                                                                                            \end{pmatrix}
.$$ 
This matrix also describes the action on the first homology group $H_1(  T^2\setminus\{*\} ,\mathbb Z )\cong\mathbb Z^2$.  The map 
$ \phi$ is orientation reversing (the matrix has determinant $-1$) and
$\phi^2$
is the monodromy of the
orientation covering of $M$, the figure eight knot exterior.

The fibration induces a presentation of the fundamental group of $M$:
$$
\pi_1(M)\cong \langle r, s, t\mid t r t^{-1}=\phi(r),\ tst^{-1}=\phi(s)\rangle
$$
where $ \langle r, s\mid \rangle =\pi_1(T^2\setminus\{*\})\cong F_2$, and
$$
\begin{array}{rcl}
 \phi_*\colon F_2 & \to & F_2\\
 r & \mapsto & s \\
 s & \mapsto & rs
\end{array}
$$
is the algebraic monodromy, the map induced by $\phi$ on the fundamental group. 

The relationship with the presentation \eqref{eqn:presentationG} of $\pi_1(M)$ from the triangulation is given by
$$
r=UV, \qquad s=VU,\qquad t= U^{-1}.
$$
Furthermore, a peripheral group is given by $\langle rsr^{-1} s^{-1}, t\rangle$, which is the group of the Klein bottle.

We use this fibered structure to compute the variety of conjugacy classes of representations.
Set $$G=\operatorname{Isom}(\mathbb H^3)\cong \mathrm{PO}(3,1) \cong \mathrm{PSL}(2,\mathbb C)\rtimes \mathbb Z_2 ,$$ 
 let $$\hom^{\mathrm{irr}}(\pi_1(M),G )$$ denote the space of \emph{irreducible} representations
 (ie~that have no invariant line in $\mathbb{C}^2$).
As we are interested in deformations, we restrict to representations $\rho$ that preserve the orientation type: $\rho(\gamma)$
is an orientation preserving isometry iff $\gamma \in \pi_1(M)$ is represented by a loop that preserves the orientation of $M$, $\forall\gamma\in \pi_1(M)$. We denote the subspace
of representations that preserve the orientation type by 
$$\hom^{\mathrm{irr}}_+(\pi_1(M),G ).$$
Let
 $$\hom^{\mathrm{irr}}_+(\pi_1(M),G )/G$$ be their the space of their conjugacy classes. 
 
\begin{Proposition}
\label{Prop:traces}
We have an homeomorphism, via the trace of $\rho(s)$:
$$
\begin{array}{rcl}
\hom^{\mathrm{irr}}_+(\pi_1(M),G )/G & \to &  \big(\{x\in \mathbb{C} \mid \vert x-1\vert =1\textrm{ and } x\neq 2\}
\big)/\!\!\sim  \\
{ [\rho]} & \mapsto & \operatorname{trace}( \rho(s))
\end{array}
$$
where $\sim  $ is the relation by complex conjugation.

In particular,  $\hom^{\mathrm{irr}}_+(\pi_1(M),G )/G$ is homeomorphic to a half-open interval.
\end{Proposition}

\begin{proof}
Let $\rho\colon \pi_ 1(M)\to G$ be an irreducible representation. 
The fibre $T^2\setminus\{*\}$ is orientable, 
so the restriction of $\rho$ to the free group $\langle r, s\mid\rangle\cong F_2$ is contained
in  $\mathrm{PSL}(2,\mathbb C)$. Furthermore, as $\langle r, s\mid\rangle$ is the commutator subgroup,
we may assume that the image $\rho( \langle r, s\mid\rangle)\subset \mathrm{SL}(2,\mathbb C)$, 
cf Heusener--Porti  \cite{HeusenerPorti04}.

 We consider the variety of characters
$X(F_2,\mathrm{SL}(2,\mathbb C))$ and the action of the algebraic monodromy $\phi_*$
on the variety of characters:
$$
\begin{array}{rcl}
 \phi^*\colon X( F_2, \mathrm{SL}(2,\mathbb C)) & \to &  X( F_2, \mathrm{SL}(2,\mathbb C))  \\
 \chi& \mapsto &  \chi\circ \phi_* 
\end{array}
$$

\begin{Lemma}
\label{lemma:inv}
The restriction of $ \hom^{\mathrm{irr}}_+(\pi_1(M),G )/G$ to 
$X(F_2, \mathrm{SL}(2,\mathbb C) )$ is contained in 
$$
\{\chi\in X(F_2, \mathrm{SL}(2,\mathbb C) ) \mid \phi^*(\chi)=\overline\chi  \}
$$
\end{Lemma}

\begin{proof}[Proof of Lemma~\ref{lemma:inv}]
Let $\rho\in  \hom^{\mathrm{irr}}(\pi_1(M),G )$. If we write
$\rho(t)= A\circ c$ for $A\in\mathrm{PSL}(2,\mathbb{C})$ and $c$ complex conjugation, from the relation
$$
t \gamma t^{-1} =\phi_*(\gamma)\qquad \forall \gamma\in F_2,
$$
we get
$$
A \overline{ \rho(\gamma) } A^{-1} =\rho(\phi_*(\gamma)) \qquad \forall \gamma\in F_2.
$$
Hence if $\rho_0$ denotes the restriction of $\rho$ to $F_ 2$, it satisfies that $\overline{\rho_0}$ and 
 $\rho_0\circ\phi_*$ are conjugate, hence they have the same character and the lemma follows. 
\end{proof}

 Lemma~\ref{lemma:inv} motivates the following computation:
 
\begin{Lemma}
\label{lemma:fixed}
We have a homeomorphism:
%
$$
\{\chi_\rho\in X(F_2, \mathrm{SL}(2,\mathbb C) ) \mid  \phi^*(\chi_\rho)=\overline{\chi_\rho} \}\cong\{x\in \mathbb{C} \mid \vert x-1\vert =1\} 
$$
by setting $x= \operatorname{trace}( \rho(s))=\chi_\rho(s)$.
\end{Lemma}

\begin{proof}[Proof of Lemma~\ref{lemma:fixed}]
First at all we describe  coordinates for $ X(F_2, \mathrm{SL}(2,\mathbb C))$. 
Let $\tau_r$, $\tau_s$ and $\tau_{rs}$ denote the trace functions, ie~$
\tau_r(\chi_\rho)=\chi_\rho(r)=\mathrm{trace}(\rho(r))
$, and similarly for $s$ and $rs$.
Fricke-Klein's theorem yields an isomorphism
$$
(\tau_{r},\tau_{s}, \tau_{rs})\colon X(F^2, \mathrm{SL}(2,\mathbb C)  )\cong \mathbb C^3
$$
(see Goldman \cite{Goldman09} for a proof).
From the relations
$$
\phi_*(r)=s, \qquad \phi_*(s)=rs,\qquad \phi_*(rs)=srs,
$$
the equality $\phi^*(\chi_\rho)=\overline{\chi_\rho}$ is equivalent to:
$$
\overline{\tau_r}=\tau_s,\qquad \overline{\tau_s}=\tau_{rs}, \qquad \overline{\tau_{rs}}=\tau_{srs}=\tau_s\tau_{rs}-\tau_r.
$$
In the  expression for $ \tau_{srs} $ we have used the relation $\mathrm{tr}(AB)=\mathrm{tr}(A)\mathrm{tr}(B)-\mathrm{tr}(AB^{-1})$ 
for $A,B\in  \mathrm{SL}(2,\mathbb C) $.
Taking $x=\tau_r=\tau_{rs}$ and $\tau_s=\overline x$, the defining equation is $x+\overline{x}=x \overline{x}$. Namely, the circle
$
\vert x-1\vert =1
$.
\end{proof}

To prove Proposition~\ref{Prop:traces}, we need to know which conjugacy classes of representations of $F^2$ are irreducible. 
By Culler--Shalen \cite{CullerShalen}, a 
character $\chi_\rho$
in  $ X(F^2, \mathrm{SL}(2,\mathbb C))$ is reducible iff  $\chi_\rho([r,s])=
\operatorname{tr}( \rho([r,s])) =2$, 
and a straightforward computation shows that this happens in the circle
$\vert x-1\vert = 1$ precisely when $x=2$. Now, let  $\rho$ be a representation of  $F^2 $ in $\mathrm{SL}(2,\mathbb C)$ whose
character $\chi_\rho$ satisfies $\phi^*(\chi_\rho)=\overline{\chi_\rho} $.
Assume  $\rho$ is irreducible, then $\rho\circ\phi_* $ and $\overline{\rho} $ are conjugate 
by a unique matrix $A\in  \mathrm{PSL}(2,\mathbb C)$:
$$
A c\rho(\gamma)c A^{-1}=A \overline{\rho(\gamma)} A^{-1}
= \rho( \phi_*(\gamma) ),\qquad \forall \gamma\in F_2,
$$
where $c$ means complex conjugation. Thus, by defining $\rho(t)= A\circ c$ this  gives a unique way to extend $\rho$ to $\pi_1(M)$. 
 
When $\chi_\rho$ is reducible, then $x=2$ and the character $\chi_\rho$ is trivial. Then either $\rho$ is trivial or parabolic. 
In any case, it is easy to check that all possible extensions to $\pi_1(M)$ yield reducible representations.
  \end{proof}

\subsection{Comparing both ways of computing deformation spaces}

We relate both ways of computing deformation spaces, via the ideal simplex and via the fibration:

\begin{Lemma}
Given a triangulated structure with parameter $w$ as in 
\eqref{eqn:proper}, the parameter $x$ of its holonomy as in 
Proposition~\ref{Prop:traces} is
 $$
 x=1+w+|w|^2
 $$
(or $x=1+\overline w+|w|^2$, because $x$ is only defined up to complex conjugation).
 \end{Lemma}

\begin{proof}
 As $r=UV$, a straightforward computation yields
 $$
 \rho(r)=\begin{pmatrix} 0 & |w|^2 \\  -\frac{1}{\vert w\vert ^2} &  1+w+|w|^2 \end{pmatrix}
 \in \mathrm{SL}(2,\mathbb C).
 $$
 Then the lemma follows from $x=\mathrm{trace}(\rho(r))$
\end{proof}

The fact that not all deformations are obtained from triangulations (Corollary~\ref{Coro:notrealized})  is illustrated in the following remark, 
whose proof is an elementary computation.

\begin{Remark}
The image of the map
$$
\begin{array}{rcl}
\{w\in\mathbb C\mid \vert w(1+w)\vert=1\}& \to & \{x\in \mathbb{C} \mid \vert x-1\vert =1\} \\
w & \mapsto & x=1+w+|w|^2
\end{array}
$$
is $ \{\vert x-1\vert =1\}\cap \{\mathrm{Re}(x)\geq \frac{3}{2}\} $, ie~the
arc of circle bounded by the image of the holonomy structure (and its complex conjugate).
See Figure~\ref{Figure:xandw}.
\end{Remark}

\begin{figure}[h]
\begin{center}
\begin{tikzpicture}[line join = round, line cap = round, scale=1.5]
 \draw [blue, thin, dashed ] (1,0) circle [radius=1];
 \draw [red, very thick ]   (2,0) arc[radius = 1, start angle= 0, end angle= 60];  ;
 \draw [red, very thick ]   (2,0) arc[radius = 1, start angle= 0, end angle= -60];  ;    
\draw(1,0)[fill=black] circle(.03);
\draw(2,0)[fill=black] circle(.03);
\draw(1.500, .866)[fill=black] circle(.03);
\draw(1.500, -.866)[fill=black] circle(.03);
 \draw(2.2,0) node{$2$};  
  \draw(1.2,0) node{$1$};  
\draw(-.7,.5)[blue] node{$\vert x-1\vert =1$}; 
 \draw(1.8,1.1) node{$ \frac{3+i\sqrt{3}}{2}$  };
\draw(1.8,-1.1) node{$ \frac{3-i\sqrt{3}}{2}$  };
 \draw(3,.5)[red] node{$x=1+w+|w|^2 $};
\end{tikzpicture}
\end{center}
 \caption{ The image of $ x=1+w+|w|^2$ in the circle $\vert x-1\vert =1$. }
\label{Figure:xandw}
\end{figure}
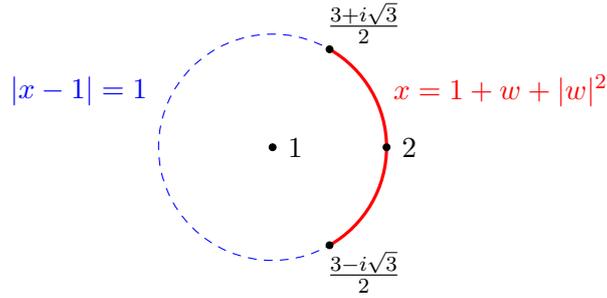

To be precise on the type of structures at the peripheral Klein bottle, 
we compute the trace of the peripheral element $[r,s]$ for each method and apply Lemma~\ref{Lemma:typeofrep}:
\begin{itemize}
 \item We compute it from the variety of representations, ie~from $x$. Using the notation of the proof of Proposition~\ref{Prop:traces}:
$$
\tau_{[r,s]}= x_1^2+x_2^2+x_3^2 - x_1x_2x_3-2= (x+\overline{x})^2-3 (x+\overline{x})-2= (x+\overline{x})((x+\overline{x})-3)- 2.
$$
The complete hyperbolic structure corresponds to $ (x+\overline{x})=3$, hence, by deforming $x$ we may have 
either $\tau_{[r,s]}> -2$ or $\tau_{[r,s]}< -2$.

\item  Next we compute it from the ideal triangulation, ie~from $w$. As $x=1+w+|w|^2$, we get
$$
\tau_{[r,s]}= 2 \operatorname{Re} (w+ w^2)\geq -2
$$
because $ \vert w+ w^2 \vert =1$. 
\end{itemize}

\begin{Remark}
As a final remark, we notice that the  path of deformations of the Gieseking manifold lifts to a path
of deformations of the figure-eight knot exterior that is the same
as the one considered
by Hilden, Lozano and Montesinos in \cite{HLMRemarkable} by deforming polyhedra.
The transition from  type I to type II of the Gieseking manifold corresponds to the 
\emph{spontaneous surgery} in \cite{HLMRemarkable}.
\end{Remark}

\bibliographystyle{plain}

\noindent \textsc{Departament de Matem\`atiques, Universitat Aut\`onoma de Barcelona, 
08193 Cerdanyola del Vall\`es, and 
Centre de Recerca Matem\`atica}

\noindent \textsf{jduran@mat.uab.cat},   
\textsf{porti@mat.uab.cat}

\end{document}